\documentclass[11pt]{article}
\usepackage{amsmath, amsthm, amssymb, amscd, mathptmx}

\usepackage{amsfonts}

\usepackage{color}
\usepackage{pstricks}

\usepackage[all]{xy}\CompileMatrices\SelectTips{cm}{12}

\theoremstyle{plain}
\newtheorem{Thm}{\sc Theorem}[section]
\newtheorem{Theorem}[Thm]{\sc Theorem}
\newtheorem{Corollary}[Thm]{\sc Corollary}
\newtheorem*{Corollary*}{\sc Corollary}

\newtheorem{Proposition}[Thm]{\sc Proposition}
\newtheorem*{Proposition*}{\sc Proposition}
\newtheorem{Lemma}[Thm]{\sc Lemma}
\newtheorem{Conjecture}[Thm]{\sc Conjecture}

\theoremstyle{definition}
\newtheorem{Definition}[Thm]{Definition}

\theoremstyle{remark}
\newtheorem{Remark}[Thm]{Remark}
\newtheorem{Example}[Thm]{Example}
\newtheorem*{Example*}{Example}
\newtheorem*{Remark*}{Remark}

\renewcommand{\AA}{{\mathbb A}}
\newcommand{\CC}{{\mathbb C}}

\newcommand{\ZZ}{{\mathbb Z}}

\newcommand{\PP}{{\mathbb P}}
\newcommand{\QQ}{{\mathbb Q}}
\newcommand{\RR}{{\mathbb R}}

\newcommand{\cO}{{\mathcal O}}

\newcommand{\Alb}{{\mathop{\rm Alb \, }}}
\newcommand{\Pic}{{\mathop{\rm Pic \, }}}
\newcommand{\GL}{\mathop{\rm GL\, }}
\newcommand{\PGL}{\mathop{\rm PGL\, }}
\newcommand{\SL}{\mathop{\rm SL\, }}

\newcommand{\Hom}{{\mathop{{\rm Hom}}}}

\newcommand{\orb}{{\mathop{\rm orb \, }}}

\newcommand{\res}{{\mathop{\rm res}}}
\newcommand{\rk}{{\mathop{\rm rk \,}}}

\newcommand{\Sym}{{\mathop{{\rm Sym}}}}

\begin{document}

\markboth {\rm Rank $3$ rigid representations}{}

\title{Rank $3$ rigid representations of projective fundamental groups}
\author{Adrian Langer, Carlos Simpson}
\date{}

\maketitle


\begin{abstract}
Let $X$ be a smooth complex projective variety with  basepoint $x$.
We prove that every rigid integral irreducible representation $\pi_1(X\! ,x)\to \SL \! (3, \CC)$
is of geometric origin, i.e., it comes from some family of smooth projective varieties. This partially generalizes
an earlier result by K. Corlette and the second author in the rank $2$ case and answers one of their questions.
 \end{abstract}

\section{Introduction}

The main examples of local systems in complex algebraic geometry are those that come from
families of varieties. Suppose $f:Z\rightarrow X$ is a smooth projective morphism of algebraic varieties. Then the
$i$-th higher direct image of the constant sheaf $\CC _Z$ is a semisimple local system $R^if_{\ast}(\CC _Z)$ on $X$.
It furthermore has a structure of {\em polarized $\ZZ$-variation of Hodge structure} with integral structure
given by the image of the map $R^if_{\ast}(\ZZ _Z)\rightarrow R^if_{\ast}(\CC _Z)$. We will view all the irreducible
direct factors of such local systems as coming from geometry.

Somewhat more generally, over a smooth
variety $X$ we say that an irreducible $\CC$-local system
$L$ {\em is of geometric origin} if there is a Zariski open dense subset $U\subset X$ and a smooth projective family $f:Z\rightarrow
U$ such that $L|_U$ is a direct factor of $R^if_{\ast}(\CC _Z)$ for some $i$. See Sections \ref{subsec-geo} and  \ref{sec-cor}
for a further discussion of possible
variants of this notion.

Local systems of geometric origin are $\CC$-variations of Hodge structure. Furthermore,
from the  $R^if_{\ast}(\ZZ _Z)$ they inherit
integral structure, in the sense that the traces of monodromy matrices are
algebraic integers. Local systems of geometric origin also have a Galois descent property when viewed as $\QQ _{\ell}$-local systems
over an arithmetic model of $X$, as described in \cite[Theorem 4]{HBLS}.

Several natural questions may be posed. For instance,
does an integral local system underlying a $\CC$-variation of Hodge structure
have geometric origin?

Recall from Mostow-Margulis rigidity theorems that many local systems naturally occurring over higher dimensional varieties are {\em rigid},
i.e. they have no nontrivial
deformations. One can easily see that rigid $\QQ _{\ell}$-local systems automatically have the
Galois descent property mentioned above. Furthermore, it is a consequence of Corlette's theorem that rigid local systems
are $\CC$-variations of Hodge structure. It is natural to formulate the following conjecture:

\begin{Conjecture}
\label{conj-rigid-motivic}
Over a smooth projective variety $X$, any rigid local system $L$ is of geometric origin.
\end{Conjecture}

This conjecture was proven for local systems on root stacks over $\PP^1$ by Katz \cite{Ka-rigid}, who gives a complete classification
and inductive description of rigid local systems in that case. For local systems of rank $2$ it was
proven by K. Corlette and the second author \cite{CS} who also prove a stronger classification result.

A consequence of the conjecture would be the following subsidiary statement:
\begin{Conjecture}
\label{conj-rigid-integral}
Over a smooth projective variety $X$, any rigid local system is integral.\footnote{Esnault and Groechenig have recently
posted a proof of this conjecture for cohomologically rigid local systems \cite{EG}.
See Remark \ref{rem-eg} for the definition of cohomologically rigid and an application.}
\end{Conjecture}

In this paper we consider the case of rank $3$, and obtain the following main theorem.

\begin{Theorem}
\label{thm-main1}
Let $X$ be a smooth complex projective variety with basepoint $x$.
Then every rigid integral irreducible representation  $\rho:\pi_1(X, x)\to \SL (3, \CC)$
is of geometric origin.
\end{Theorem}

Our techniques will not address the question of proving that
a local system is integral, so we have included integrality as a hypothesis. Consequently, our theorem does not provide a complete answer to Conjecture \ref{conj-rigid-motivic} in the rank $3$ case, but it
does show for rank $3$ that Conjecture \ref{conj-rigid-integral} implies Conjecture \ref{conj-rigid-motivic}.

On the other hand, for varieties with no  symmetric differentials  in some range, Conjecture \ref{conj-rigid-integral} was proven by B. Klingler in \cite{Kl}. Together with Theorem \ref{thm-main1} his results imply the following
corollary:

\begin{Corollary}
Let  $X$ be a smooth projective variety with $H^0(X, \Sym ^i\Omega_X)=0$ for $i=1,2,3$. Then any
representation $\pi_1(X,x)\to \GL (3, \CC)$ is of geometric origin.
\end{Corollary}

\medskip

We now explain the main ideas of the proof of Theorem \ref{thm-main1}. We
say that a variation of Hodge structure (VHS for short) is of {\em weight one}
if its Hodge types are contained in the set $\{ (1,0), (0,1)\}$.
It is well known that a polarizable weight one $\ZZ$-VHS comes
from an algebraic smooth family of abelian varieties. Therefore, if we can construct such a VHS then we obtain an
algebraic family and the underlying local system is of geometric origin. A refined version of this observation,
which one can learn for example from Deligne's \cite{Del}, goes as follows. Start with a local system $L$ of complex
vector spaces, but assume that $L$ is rigid. Then it is defined over some algebraic number field $K\subset \CC$. If we also assume  that
$L$
is integral, then Bass-Serre theory says that there is a local system of  projective $\cO _K$-modules $L_{\cO _K}$
with
$$
L\cong L_{\cO _K} \otimes _{\cO _K} \CC .
$$
The local system $L_{\cO _K}$  may be viewed as a local system of abelian groups, hence of (free) $\ZZ$-modules. Now
$$
L_{\cO _K} \otimes _{\ZZ} \CC = \bigoplus _{\sigma : K\rightarrow \CC} L^{\sigma},
$$
where $L^{\sigma}$ is the $\CC$-local system obtained by extension of scalars using the embedding $\sigma$.
In order to give our $\ZZ$-local system a structure of polarized weight-one $\ZZ$-VHS, we should therefore give each complex local
system $L^{\sigma}$ a structure of polarized $\CC$-VHS of weight one. Adjusting these together in order to have the
required properties to get a family of abelian varieties, will be discussed in more detail in Section \ref{sec-cor}.

Our original local system $L$ is
one of the factors $L^{\sigma _0}$ corresponding to the initially given embedding $\sigma _0: K \hookrightarrow \CC$.
The plan to show that $L$ is of geometric origin is therefore to try to put weight-one VHS's on each $L^{\sigma}$.
Assuming $L$ is irreducible, so all these local systems are irreducible, if they have structures of VHS then these structures are
unique up to translation of Hodge types. The existence of weight-one structures is therefore a property of the $L^{\sigma}$.

In the case ${\rm rk}(L)=2$, this property is automatic. Indeed the only two possibilities for the Hodge numbers are
a single Hodge number equal to $2$, a case we denote $(2)$, or two adjacent Hodge numbers equal to $1$, a case we
denote $(1,1)$. Notice that if the nonzero Hodge numbers are not adjacent then the Kodaira-Spencer map (Higgs field)
must be zero and the system becomes reducible. This explains in a nutshell the procedure that was used in \cite{CS} to treat
the rigid and integral local systems of rank $2$.

In the case of local systems of rank $3$, there are four possibilities for the type of a VHS: either $(3)$ corresponding to unitary ones,
or $(1,2)$ or $(2,1)$ corresponding to weight-one VHS (with Hodge numbers $h^{1,0}=1, h^{0,1}=2$ or $h^{1,0}=2, h^{0,1}=1$ respectively); or the last case $(1,1,1)$ when the Hodge bundles are three line bundles. The unitary type $(3)$ can be considered
as having weight one in two different ways. Thus we have the following lemma which is implicit in proof of \cite[Theorem 8.1]{CS}:

\begin{Lemma}
Suppose $V$ is an irreducible rank $3$ local system on a smooth
quasi-projective variety $X$. If $V$ is rigid then it underlies either a
complex VHS of weight one or a complex VHS of type $(1,1,1)$.
\end{Lemma}

The reasoning up until now has been well-known and standard.
The main idea of the present paper is that we can use substantial arguments in birational geometry to rule out the case
of VHS type $(1,1,1)$, unless some special behaviour occurs namely factorization through a curve. In turn,
the case of factorization through a curve can be shown, again by standard arguments using Katz's theorem,
to lead to local systems of geometric origin under the hypothesis
of rigidity. We show the following main theorem.
See Theorem \ref{main_thm_on_type_(1,1,1)} for a more detailed version of part (1) of the conclusion.

\begin{Theorem}
\label{thm-main2}
Let $X$ be a smooth complex projective variety with basepoint $x$
and let $\rho:  \pi_1(X,x)\to \SL (3,\CC)$ be an irreducible representation coming
  from a complex variation of Hodge structure of type $(1,1,1)$.  Let
  $V_{\rho}$ denote the corresponding local system.  Then one of the
  following holds:
\begin{enumerate}
\item the image of $\rho$ is not Zariski-dense in $\SL (3,\CC)$; or
\item $\rho$ projectively factors through an orbicurve.
\end{enumerate}
\end{Theorem}

Once we have this theorem,  the proof of Theorem \ref{thm-main1} follows the outline described above: the case of factorization
through a curve is treated on the side, and otherwise the theorem rules out having a VHS of type $(1,1,1)$. Therefore, all of
our local systems $L^{\sigma}$ underly VHS of weight one, and these may be put together into a polarized $\ZZ$-VHS of weight one
corresponding to a family of abelian varieties. The details are described in Section \ref{sec-cor}.

We note that Theorem \ref{thm-main2} generalizes Klingler's result \cite[Proposition 3.3]{Kl1} which had as hypothesis that the Neron-Severi group be of rank $1$.

\medskip

Let us now explain briefly how we can use arguments from birational geometry to
rule out the case of VHS of type $(1,1,1)$. Such a VHS corresponds to a Higgs bundle of the form
$$
E= E^{2,0}\oplus E^{1,1}\oplus E^{0,2}
$$
with $E^{p,q}$ line bundles, and the Higgs field consists of $\theta : E^{2,0}\rightarrow E^{1,1}\otimes \Omega ^1_X$
and $\theta : E^{1,1}\rightarrow E^{0,2}\otimes \Omega ^1_X$. In terms of the VHS $(V, F^{\cdot}, \nabla )$,
the line bundles are the Hodge bundles
$E^{p,q} = F^pV /F^{p+1}V$
and the $\theta$ are the Kodaira-Spencer maps induced by
$\nabla$.  Since the $E^{p,q}$ are line bundles, locally $\theta$ looks like a collection of
sections of $\Omega ^1_X$ and the integrability condition $\theta \wedge \theta = 0$ implies that these local sections
are proportional where nonvanishing. They generate a saturated sub-line bundle $M\subset \Omega ^1_X$.

For simplicity, let us here assume that $X$ is a surface.
Consider the line bundles $L_1= E^{2,0}\otimes (E^{1,1})^{\ast }$ and $L_2=E^{1,1}\otimes  (E^{0,2})^{\ast}$. The
sections $\theta$ may be viewed as inclusions $L_i\hookrightarrow M\subset \Omega ^1_X$. Write
$$
M=L_1(B_1) = L_2(B_2)
$$
with $B_i$ effective divisors. Bogomolov's lemma (see \cite[Theorem 4]{Bog}) says that $M$ cannot be big. We know that
the Higgs bundle $(E, \theta )$ is stable, and its rational Chern classes vanish. We are able to conclude several fairly
strong properties:

\begin{itemize}

\item The line bundles $L_i$ both lie on the same line in $NS (X)_{\QQ}$;

\item At least one of $L_i$ is nef of strictly positive degree; and

\item $L_i^2=0$.

\end{itemize}

That happens in Section \ref{structure=(1,1,1)}. The technique is to play off the numerical properties given by stability and vanishing of Chern
classes, against the fact that we have effective divisors $B_i$ such that $L_i+B_i$ are contained in $\Omega ^1_X$ and
not big by Bogomolov's lemma.

Then in Section \ref{sec-geometry}, assuming for example that $L_1$ is nef of strictly positive degree,
we write $L_2=a.L_1$ in the rational Neron-Severi group, with $a\in \QQ$, and we try to get information about $a$.
Notice that having $L_1^2=0$ and $L_1\hookrightarrow \Omega ^1_X$ allows us to create the rank $2$
Higgs bundle $\cO _X \oplus L_1$ corresponding to a projectively flat connection. We would like to
view $E$ as being the symmetric square of this rank $2$ bundle.
In other words, we would like to show that $a=1$ and moreover $L_1=L_2$.

Part of the difficulty is that there exist VHS's over a curve where $a$ could be somewhat arbitrary, including being
negative. Therefore, our strategy is to
show that either $a=1$, or else our VHS factors through a map to a curve. In the case $a=1$, again with some further
arguments assuming that there is no factorization through a curve, we finally conclude that $L_1=L_2$ and
indeed our VHS was a symmetric square of a rank $2$ local system. These arguments use various aspects of the theory
of factorization of representations. We include a proof that several different statements of the property ``factors through a curve''
are equivalent, in the Appendix. The proof of Theorem \ref{thm-main2} is concluded at the end of Section \ref{sec-geometry}.

Section \ref{sec-someremarks} contains some remarks about the possible extension of this discussion to VHS's of type $(1,1,\ldots , 1)$ in
higher ranks, and Section 6 derives the corollaries about local systems of geometric origin leading to the proof of Theorem
\ref{thm-main1}.

\subsection{Notation}
\label{sec-notation}

Let us recall that a Higgs bundle is a pair $(E,\theta)$ consisting
of a locally free $\cO_X$-module and an $\cO_X$-linear map $\theta:
E\to E\otimes \Omega_X$ such that $\theta\wedge \theta=0$.  A
\emph{system of Hodge bundles} is a Higgs bundle $(E,\theta)$ with a
decomposition $E=\bigoplus E^{p,q}$ such that $\theta$ maps $E^{p,q}$
into $E^{p-1,q+1}\otimes \Omega_X$.  We say that a system of Hodge
bundles $(E, \theta)$ is \emph{of type $(i_0,...,i_n)$} if
 $$\rk E^{p,q}=
\left\{\begin{array}{ll}
i_q & \quad \hbox{ if } p+q=n,\\
0& \quad \hbox{ if }p+q\ne n.\\
\end{array}
\right.
$$
and all the maps $\theta: E^{m-q,q}\to E^{m-q-1,q+1}\otimes \Omega_X$
are non-zero for $q=0,..., n-1$. Similarly, replacing $\Omega_X$ with
$\Omega_X(\log D)$ one can define logarithmic systems of Hodge
bundles.

\medskip

In the following we make the convention that complex variation of
Hodge structure (VHS) always means polarizable VHS.  Giving an
irreducible VHS on a smooth complex projective variety is equivalent
to giving a stable system of Hodge bundles with vanishing rational
Chern classes.

We say that a complex VHS is \emph{of type $(i_0,...,i_n)$} if
the corresponding system of Hodge bundles  is of type $(i_0,...,i_n)$.

\medskip Let $X$ be a smooth complex projective variety. Let us
  recall that a line bundle is called \emph{nef} if it has a
  non-negative degree on every irreducible projective curve in
  $X$. Let us recall that after Miles Reid, ``nef'' stands for
  ``numerically eventually free''. A line bundle $L$ is called
  \emph{big} if its Iitaka's dimension $\kappa(L)$ is equal to the
  dimension of $X$.

A $\QQ$-divisor $D$ is called \emph{effective} if it can be written as
$D=\sum a_iD_i$, where $D_i$ are prime divisors and $a_i$ are
non-negative rational numbers.  We write $D_1\ge D_2$ for two
$\QQ$-divisors $D_1, D_2$, if their difference $(D_1-D_2)$ is an
effective $\QQ$-divisor.

If $L_1$ and $L_2$ are line bundles then we write $L_1\ge L_2$ if $H^0(L_1\otimes L_2^{-1})\ne 0$.

\medskip

For various other definitions and properties of line bundles we refer
the reader to \cite[Theorem 2.2.16]{Laz}.  Let us just mention that if
$L$ is nef then it is big if and only if $L^{\dim X}>0$.

\medskip A morphism $f:X\to Y$ between
smooth quasi-projective
    varieties (or just orbifolds) is called a \emph{fibration} if it
  is surjective and the fibers of $f$ are connected.

A morphism $f:X\to Y$ is called an \emph{alteration} if $X$ is smooth and
$f$ is proper, surjective and generically finite.

\medskip

In the following we often abuse notation and we do not distinguish between
an algebraic variety and the underlying analytic space. For example,
the fundamental group $\pi_1(X)$ and cohomology groups $H^{*}(X, \QQ)$
of an algebraic variety $X$ always denote the corresponding notion for
the underlying analytic space. Similarly, essentially all varieties,
bundles and maps are algebraic. In the projective case this convention
is harmless due to GAGA type theorems but in general this could lead
to some confusion so in case we use non-algebraic structures we
explicitly say so (see the use of analytic maps in the proof of
Theorem \ref{main_thm_on_type_(1,1,1)}).

\medskip

Let $X$ be a smooth complex quasi-projective variety. Let us fix a smooth projective compactification $\bar X$ of $X$
such that $D=\bar X-X$ is a simple normal crossing divisor. Let $\{\gamma_i\}$ be loops going around the irreducible components of $D$.
Let  $\rho: \pi_1(X,x)\to G$ be a representation of $\pi_1(X,x)$ in some reductive group $G$ and let us assume that all $\rho(\gamma_i)$ are quasiunipotent. Let us denote by $C_i$
the closure of the conjugacy class of $\rho (\gamma_i)$.
We say that  $\rho: \pi_1(X,x)\to G$ is \emph{rigid} if it represents an isolated point in the moduli space $M(X, G, \{C_i\})$ of representations of $\tau: \pi_1(X,x)\to G$ with conjugacy classes of $\tau (\gamma _i)$ contained in $C_i$.

\section{Preliminaries}

\label{sec-prelim}

\subsection{Homotopy exact sequence}\label{homotopy-subsection}

Let $C$  be a smooth complex  quasi-projective curve and let
$X$ be a
smooth  complex  quasi-projective  variety.   Let  $f: X\to  C$  be  a
fibration and let $c$  be a closed point of $C$. Let  us recall that the
\emph{multiplicity}  of a  fiber  $F=\sum a_i  F_i$,  where $F_i$  are
irreducible components of $F$, is equal to the greatest common divisor
of the coefficients $a_i$.  Let $p_1,...,p_k$ be all the points of $C$
at which  the fiber of $f$  is multiple (i.e., it  has multiplicity at
least $2$) and  let $m_j$ denote the multiplicity of  the fiber of $f$
over $p_j$. Let $c$ be a point  of $C$ over which the fiber of $f$ has
multiplicity $1$.

We define the \emph{orbifold fundamental group} $ \pi_1^{\orb}(C _f,
c)$ of $C$ with respect to $f$ as the quotient of $\pi _1(C-\{p_1,...,p_k\},  c)$
by the normal subgroup generated by all the elements of the form $\gamma_j^{m_j}$,
where $\gamma_j$ is a simple loop going around the point $p_j$.

\medskip

The following theorem is well-known:

\begin{Theorem} \label{Xiao} Let $c\in
  C-\{p_1,...,p_k\}$ and let $x$ be a closed point of the fibre
  $X_{c}$ of $f$ over $c$. Let us assume that either $f$ is proper or $c\in C$ is general.
  Then the following sequence of groups is  exact:
$$\pi_1(X_{ c}, x)\to \pi_1(X,  x)\to \pi_1^{\orb}(C _f,  c)\to 1.$$
In particular, if we also assume that $f$ has no multiple fibers then the sequence
$$\pi_1(X_{ c},  x)\to \pi_1(X,  x)\to \pi_1(C,  c)\to 1$$
is exact.
\end{Theorem}

\begin{proof} If all the fibers of $f$ have at least one smooth point
  then the theorem is due to M. Nori \cite[Lemma 1.5]{No} (for general
  fibre) with a small improvement due to O. Debarre \cite[Lemma
  8.11]{De} allowing to deal with arbitrary fibers in the proper case.
  Nori's proof with appropriate changes as in \cite{Xi} works also for general fibers
  even without assumption that all the fibres have at least one smooth point.
  In the general proper case, the assertion comes from \cite[Lemma 1, Lemma 2 and Lemma
  3]{Xi} in the surface case, but Xiao's proof also works in the
  general case. Alternatively, for smooth fibers the general case can
  be reduced to the surface case by the Lefschetz hyperplane
  theorem. Namely, if $\dim X\ge 3$ then by Bertini's theorem we can
  find a very ample divisor $H$ such that both $X_c\cap H$ and $H$ are
  smooth and connected. Since, by the Lefschetz hyperplane theorem,
  the map $\pi_1(X_c\cap H)\to \pi_1(X_c)$ is surjective and the map
  $\pi_1(H)\to \pi_1(X)$ is an isomorphism, exactness of the homotopy
  sequence for $f$ follows from exactness of the homotopy sequence for
  $f|_H:H\to C$.  Now the required assertion follows from the surface
  case by induction.
\end{proof}

\begin{Remark}
If $f$ is not proper then the homotopy sequence from the above theorem need not be exact even
if we assume that all the fibers of $f$ are smooth. One can easily construct examples when this sequence
is not exact by blowing up a smooth surface fibered over a curve and removing non-exceptional components
in the fiber over the blown-up point.
\end{Remark}

\subsection{Intersection pairing} \label{intersection-pairing}

Let $X$ be a smooth complex projective variety of dimension $d\ge 2$.
Let $A$ be a fixed very ample divisor on $X$. Then we use intersection
pairing on $\QQ$-divisors given by
$$D_1.D_2:=D_1\cdot D_2\cdot A^{d-2}.$$
We will often use the fact that if $D.A=D^2=0$ then the class of $D$
in $H^2(X, \QQ)$ is $0$.  To prove that note that if $Y\in |A|\cap
...\cap |A|$ denotes a general complete intersection surface in $X$
then the class of $D|_Y$ in $H^2(Y, \QQ)$ is zero by the Hodge index
theorem. But by Lefschetz' hyperplane theorem the restriction $H^2(X,
\QQ)\to H^2(Y, \QQ)$ is injective so the class of $D$ in $H^2(X, \QQ)$
is also zero.

In the following we write $D_1\equiv D_2$, if for every $\QQ$-divisor
$D$ we have $D_1.D=D_2.D$. In that case Lefschetz' hyperplane theorem
implies that the class of $(D_1-D_2)$ in $H^2(X, \QQ)$ is equal to
zero.  $L_1\equiv L_2$ for line bundles $L_1$ and $L_2$ denotes
equality $c_1(L_1)\equiv c_1(L_2)$ in $H^2(X, \QQ)$. In that case
$L_1\otimes L_2^{-1}\in \Pic ^{\tau}X$.

\medskip

If $\dim X=2$ then an effective $\QQ$-divisor $N=\sum a_iN_i$, where
$a_i>0$ and $N_i$ are prime, is called \emph{negative} if the
intersection matrix $[N_i.N_j]$ is negative definite.  It is called
\emph{semi-negative} if the matrix $[N_i.N_j]$ is negative
semi-definite. By convention, the zero divisor is both negative and
semi-negative.

\subsection{Zariski semi-decomposition}

Let $X$ be a smooth complex projective surface. We say that the degree
of a $\QQ$-divisor $L$ on $X$ is \emph{strictly positive} if for every
ample divisor $H$ we have $L.H>0$.

\begin{Lemma} \label{nefness} Let $L$ be a $\QQ$-divisor on
  $X$. Assume that $L.A>0$ for some ample $A$ and $L^2\ge 0$. If $L$
  is not big then $L$ is nef, $L^2=0$ and the degree of $L$ is
  strictly positive.
\end{Lemma}

\begin{proof}
  Let us first note that $L$ is pseudoeffective, i.e., its
  intersection with any nef divisor is non-negative.  Indeed, if
  $L.H<0$ for some nef $H$ then there exists a positive rational
  number $a$ such that $L.(A+aH)=0$. Since $A+aH$ is ample, by the
  Hodge index theorem we get $L^2<0$, a contradiction.

Now recall that pseudoeffective divisors admit the so called Zariski
decomposition. This means that $L$ can be written as a sum $P+N$,
where $P$ is a nef $\QQ$-divisor and $N$ is a negative effective
$\QQ$-divisor with $P.N=0$. Since $L$ is not big, $P$ is also not big
which implies that $P^2=0$. But in this case if $N\ne 0$ then $L^2=N^2<0$, a
contradiction. Therefore $L=P$ so $L$ is nef with $L^2=0$.

To prove that the degree of $L$ is strictly positive it is sufficient
to show that there are no ample divisors $H$ with $L.H=0$. But if such
a divisor $H$ exists then by the Hodge index theorem $L$ is
numerically trivial contradicting inequality $L.A>0$.
\end{proof}

\medskip

\begin{Lemma}\label{Zariski-semi-decomposition}
Let $C$ be a $\QQ$-divisor on $X$. Assume that $C=L+B$, where $L$ is nef and $B$ is
effective. If $C$ is not big then the following conditions are satisfied:
\begin{enumerate}
\item $L^2=0$,
\item $B$ is semi-negative,
\item $L.B=0$. In particular, if we write $B=\sum a_i C_i$, where $a_i>0$ and $C_i$ are prime divisors,
then $L.C_i=0$ for every $i$.
\end{enumerate}
In this case we say that $C=L+B$ is a \emph{Zariski
semi-decomposition} of $C$.
\end{Lemma}

Note that unlike the usual Zariski decomposition, Zariski semi-decomposition need
not be unique, i.e., one divisor can have many different Zariski semi-decom\-po\-si\-tions.

\begin{proof}
Since $L$ is nef we have $L^2\ge 0$. If $L^2>0$ then $L$ is big.
But then $C$ is big contradicting our assumption. This proves the
first assertion.

To prove the second one, it is sufficient to take any combination
$\sum \alpha_i C_i$ with rational numbers $\alpha_i\in [0,a_i]$
and prove that $(\sum \alpha_i C_i)^2\le 0$.

Let us take a rational number $\beta\in [0,1]$. Since $(\beta L+\sum
\alpha_i C_i)\le C$ and $C$ is not big, the $\QQ$-divisor $(\beta
L+\sum \alpha_i C_i)$ is not big. But $(\beta L+\sum \alpha_i
C_i).A\ge 0$, so we have $(\beta L+\sum \alpha_i C_i)^2\le 0$. Putting
$\beta=0$, we get the required inequality.

Putting $\beta=1$ we see that for any $\epsilon \in [0,1]$, we have
$(L+\epsilon B)^2\le 0$. Since $L^2=0$, this gives $2L.B+\epsilon
B^2\le 0$. Passing with $\epsilon$ to $0$ we get $L.B\le 0$. But $L$
is nef and $B$ is effective so we get $L.B=0$.
\end{proof}

{

\subsection{Lifting of representations of the fundamental group}

Let $G/\CC$ be a connected, semisimple group and let $\Gamma$ be a
group.  A homomorphism of groups $\rho : \Gamma \to G$ is called
\emph{irreducible} if the Zariski closure of the image of $\Gamma$ is
not contained in any proper parabolic subgroup of $G$. In case $G=\SL
(n, \CC)$, this notion is equivalent to the usual notion of an
irreducible representation.

Let us recall the following definition (cf. \cite[2.14]{dJ}):

\begin{Definition}
A homomorphism  $\rho: \Gamma\to G$ is called \emph{Lie irreducible} if it is irreducible and remains so
on all finite index subgroups of $\Gamma$.
\end{Definition}

Let $\rho: \Gamma\to G$ be an irreducible homomorphism and let
$G_{\rho}\subset G$ be the Zariski closure of the image of
$\rho$. Then $\rho$ is Lie irreducible if and only if the inclusion of
the connected component $G_{\rho}^0$ of $G_{\rho}$ into $G$ is
irreducible. In that case $G_{\rho}$ is semisimple and its centralizer
in $G$ is finite.

\medskip

The following lemma is a stronger version of \cite[Theorem 3.1]{Co}:

\begin{Lemma} \label{lifting} Let $X$ be a smooth, complex,
  quasi-projective variety and let $\rho :\pi_1(X, x)\to G$ be a
  homomorphism into a complex, linear algebraic group $G$ and let
  $\tilde G\to G$ be a central isogeny.  Then there exists a finite
  surjective morphism $p:Z\to X$ such that $p^*\rho: \pi_1(Z, z)\to G$
  lifts to $\pi_1(Z, z)\to \tilde G$. Moreover, for any completion
  $X\subset \bar X$ there exists an alteration $\bar p: \bar Z\to \bar
  X$, such that $\bar Z$ is smooth and projective, $D=\bar Z-Z$ is a
  simple normal crossing divisor and the restriction $p=\bar
  p|_{p^{-1}X}:Z=p^{-1}(X)\to X$ is such that $p^*\rho: \pi_1(Z, z)\to
  G$ can be lifted to $\pi_1(Z, z)\to \tilde G$.
\end{Lemma}

\begin{proof}
Let $A$ be the finite abelian group defined by the short exact sequence
$$1\to A \to \tilde G\to  G\to 1.$$
Since $A$ is contained in the center of $\tilde G$, this sequence induces an exact sequence
$$ H^1(X,  \tilde G)\to  H^1(X, G ) \mathop{\to}^{\delta} H^2 (X, A).$$
By the comparison theorem, for a finite abelian group $A$ we have a natural isomorphism
$H^2 (X, A)\simeq H^2 _{\acute{e}t}(X, A)$ of the complex cohomology group with an \'etale cohomology group.
By \cite[Theorem 1.1]{Bh} there exists a finite surjective morphism $p:Z\to X$ such that the image of
the obstruction class $\delta ([\rho])$ in $ H^2 _{\acute{e}t}(Z, A)$ vanishes. This shows that
$p^*\rho: \pi_1(Z, z)\to G$ can be lifted to a homomorphism $\rho_Z: \pi_1(Z, z)\to \tilde G$.

The second part of the lemma follows from the first one and from
existence of log resolution of singularities of the normalization of
$\bar X$ in the function field of $Z$.
\end{proof}

\medskip

Let us recall that a linear algebraic group is called
\emph{almost-simple}, if it does not contain any Zariski closed,
connected normal subgroups of positive dimension. Clearly, an almost-simple group
is semi-simple and connected.

\begin{Proposition}\label{weak-general}
Let $X$ be a smooth complex quasi-projective variety and let $G/\CC$ be a connected, semisimple group.
Let $\rho :\pi_1(X, x)\to G$ be an irreducible homomorphism. Then one of the following holds:
\begin{enumerate}
\item There exists a finite \'etale covering $p:Z\to X$ such that
  $p^*\rho: \pi_1(Z, z)\to G$ is not irreducible.
\item There exists a finite \'etale covering $p:Z\to X$ such that the
  Zariski closure of the image of $p^*\rho$ in $G$ is almost-simple and irreducible.
\item There exists a finite surjective morphism $p:Z\to X$ and simple, simply connected groups $H_1,...,H_m$, $m\ge 2$, such
  that $p^*\rho: \pi_1(Z, z)\to G$ factors through a homomorphism
  $\tau: \pi_1(Z,z)\to H_1\times ...\times H_m$ with a Zariski dense
  image.
\end{enumerate}
\end{Proposition}

\begin{proof}
  There exists a finite \'etale covering $Z\to X$ such that
  $\pi_1(Z,z)$ maps into the connected component $H=G_{\rho}^{0}$ of
  the Zariski closure of the image of $\pi_1(Z,z)$ in $G$.  If
  $H\subset G$ is not irreducible then we are in the first case. So
  let us assume that $H\subset G$ is irreducible.  Then $H$ has a
  canonical decomposition into an almost direct product of its
  almost-simple factors (see \cite[Theorem 22.10]{Bo}). More
  precisely, let $H_1,...,H_m$ be the minimal elements among the
  closed, connected normal subgroups of $H$ of positive
  dimension. Then all $H_i$ are almost-simple and the isogeny
  $H_1\times ...\times H_m\to H$ is central (with finite kernel).  If
  $m=1$ then we are in the second case. If $m\ge 2$ then we can pass
  to the universal covering of the product $H_1\times ...\times H_m$
  and we are in the last case by Lemma \ref{lifting}.
\end{proof}

\medskip

\begin{Remark}\label{weak-general-fibration}
The second case of the above proposition is particularly interesting in presence of fibrations.
More precisely, if $f:X\to Y$ is a fibration, $G$ is almost-simple and $\rho :\pi_1(X, x)\to G$ has Zariski dense image,
then the image of the restriction of $\rho$ to a smooth fiber of $f$ is either Zariski dense in $G$ or it is finite
(see \cite[Proposition 2.2.2]{Zuo}). In the second case, after some blowing up and taking a finite \'etale cover, one can factor
$\rho$ through the induced fibration (see \cite[Lemma 2.2.3]{Zuo}).
\end{Remark}
}

\subsection{Representations of geometric origin}
\label{subsec-geo}

Let $X$ be a smooth complex quasiprojective variety and let $\rho: \pi_1(X,x)\to \SL (n, \CC)$
be a semisimple representation. Let us recall the following definition from the Introduction:

\begin{Definition}
\label{geometricorigin}
  We say that a semisimple representation $\rho$ and its associated local system $V$ are {\em of geometric origin} if there exists a dense
  Zariski open subset $U\subset X$ and a smooth projective morphism
  $f:Z\rightarrow U$ such that $V|_U$ occurs as a subquotient of the local system $R^if_{\ast}(\CC _Z)$ for some $i$.
\end{Definition}

\begin{Proposition}
\label{prop-1-9}
If $V_1$ and $V_2$ are local systems then $V_1\oplus V_2$ is of geometric origin if and only if $V_1$ and $V_2$ are. If $V_1$ and $V_2$ are of geometric origin then so is $V_1\otimes V_2$. A rank $1$ local system is of geometric origin if and only if it is of finite order.

If $p:Y\rightarrow X$ is a generically surjective map between irreducible varieties, and $V$ is a semisimple local system on $X$, then $V$ is of geometric origin if and only if $p^{\ast}(V)$ is of geometric origin.

Suppose that $\rho_1$ and $\rho _2$ are two irreducible representations with determinants of finite order, whose projectivizations are isomorphic. Then $\rho_1$ is of geometric origin if and only if $\rho _2$ is.
\end{Proposition}

\begin{proof}
Given families for $V_1$ and $V_2$ we can take a disjoint union of products
with fixed varieties for $V_1\oplus V_2$. One may also stay within the world of fibrations
with connected fibers by then embedding the resulting family in a
big projective space and blowing up along that locus.
If $V_1\oplus V_2$ is of geometric origin then $V_1$ and $V_2$ are so by definition. If $V_1$ and $V_2$ are of geometric origin, taking the fiber product of the families results, by the K\"unneth formula, in a family for $V_1\otimes V_2$.

A rank $1$ system of finite order is a direct factor in the
$0$-th direct image of a finite ramified cover.
In the other direction, let $V$ be a rank $1$ local system of geometric origin.
Let $f:Z\to U\subset X$ be the corresponding map. Note that $R^if_*\CC_Z$ comes with an integral structure given by the image of $R^if_*\ZZ_Z\to R^if_*\CC_Z$. Hence the eigenvalues of all monodromy transformations are algebraic integers, so
any direct summand
of $R^if_*\CC_Z$ is integral. The action of $\mathop{\rm Gal} \, (\CC/\QQ)$ takes one direct summand to another. Now existence of a polarization implies that any rank $1$ direct summand is unitary. Thus, $V$ and all its
$\mathop{\rm Gal} \, (\CC/\QQ)$-conjugates are integral and unitary. But by Kronecker's theorem, an algebraic integer
all of whose conjugates have absolute value $1$ is a root of unity, so $V$ is of finite order (not only on $U$ but also on $X$).

Suppose $p:Y\rightarrow X$ is a generically surjective map and $V$ a local system on $X$. If $V$ is of geometric origin on $X$ then
since the image of $p$ intersects the open set $U\subset X$ on which
$V$ is defined, we get a family over $p^{-1}(U)$ showing that $p^{\ast}(V)$ is of geometric origin.

In the other direction, suppose $p^{\ast}(V)$ is of geometric origin coming from a family $f:Z\rightarrow U$ with $U\subset Y$.
Choose a subvariety $Y'\subset U$ which is quasi-finite over $X$.
By replacing $Y'$ by a Zariski open subset, we may assume that $p':Y'\rightarrow U'$ is a finite \'etale map to a dense open subset {$U'\subset X$}. Let $Z'\subset Z$ be the inverse image of $Y'$,
and denote by $f':Z'\rightarrow Y'$ the map. {Under our
suppositions, it is smooth and projective, hence also} the composed map
$$
Z' \stackrel{f'}{\rightarrow} Y'\stackrel{p'}{\rightarrow} U'
$$
is smooth and projective. By the
decomposition theorem, $R^i (f')_{\ast}(\CC _{Z'})$ is semisimple
so a subquotient may be viewed as a direct factor. Thus,
we may assume given an injective map
$(p')^{\ast}(V)\rightarrow
R^i (f')_{\ast}(\CC _{Z'})$ which gives by composition, in turn, an injective map
$$
V|_{U'} \rightarrow (p')_{\ast}(p')^{\ast}(V) \rightarrow
R^i (p'f')_{\ast}(\CC _{Z'}),
$$
showing that $V$ is of geometric origin.

For the last part, suppose given two irreducible representations $\rho_1$ and $\rho _2$ whose determinants are of finite order. Let
$V_1$ and $V_2$ respectively denote the corresponding local systems.  Suppose the projectivizations of $\rho_1$ and $\rho _2$ are isomorphic, and fix bases for $(V_1)_x$ and $(V_2)_x$ at the basepoint $x\in X$
that are compatible with the isomorphism of projectivizations.
For any $\gamma \in \pi _1(X,x)$ we may then write
$$
\rho _2(\gamma ) = a(\gamma ) \rho _1(\gamma ),
$$
with $a(\gamma )$ a diagonal matrix. Now $a:\pi _1(X,x)\rightarrow \CC ^{\ast}$ is a rank $1$ representation corresponding to a rank $1$ local system $A$, and we have
$V_2=A\otimes V_1$. If $r={\rm rk}(V_1)$ then
$$
{\rm det}(V_2) = A^{\otimes r}\otimes {\rm det}(V_1).
$$
From the hypothesis, $A^{\otimes r}$ is of finite order, so $A$ is of finite order. Hence, $A$ is of geometric origin. It follows that $V_2$ is of geometric origin if and only if $V_1$ is.
\end{proof}

Concerning the second paragraph, it is also true that any pullback of a system of geometric origin, is of geometric origin. For the case of pullback along a map that goes into the
complement of the open set $U$ over which the smooth family is given, this requires using the theory of perverse sheaves and applying the decomposition theorem; we do not treat that here.

The converse to the part about $V_1\otimes V_2$ looks like an interesting
question: if $V_1$ and $V_2$ are irreducible local systems with trivial determinant such that $V_1\otimes V_2$ is of geometric origin then are $V_1$ and $V_2$ of geometric origin?

\medskip

Our definition of local systems of geometric origin is rather large,
and one can envision a couple of strengthenings.
The following remarks on this point are not needed in the rest of the paper:
we refer everywhere only to the notion of
Definition \ref{geometricorigin}, and leave open the question of
obtaining the stronger properties we now mention.

First, one could ask for $U=X$, that is to say one could ask for a smooth family extending over $X$. We might say in that case that $V$ is {\em smoothly of geometric origin}. This looks to be a very strong condition, and it is difficult to predict when it should be satisfied. There are natural examples of local systems coming from families of varieties, such that
the local system extends to a larger subset than the subset of definition of the smooth family. Indeed, any even-dimensional Lefschetz pencil has order two monodromy around the singularities, so after pulling back to a twofold cover of the base ramified at the singular points, the local system extends. It is highly unclear whether to expect the existence of a smooth globally defined family realising the same monodromy representation.

In particular, Katz's theorem that rigid local systems over the orbifold projective line are of geometric
origin,
does not give smooth geometric origin. The discussion of the notion of geometric origin in \cite{To2} should probably be modified accordingly.

Another point is the question of whether the
projector to our direct factor is
motivic. Say that a local system $V$ is {\em strongly of geometric origin} if there exists a proper morphism $f:Z\rightarrow X$ and an algebraic cycle representing a projector $\pi : Rf_{\ast}(\ZZ _Z)\rightarrow Rf_{\ast}(\ZZ _Z)$ such that $V$ is the image of $\pi$.

One would need at least the Hodge conjecture to go from our (weak) notion of geometric origin to this notion of strong geometric origin, but in fact it does not even seem clear that the Hodge conjecture would be sufficient for that.

For instance one would like to answer the following question
(a punctual analogue of the converse question for tensor products of
local systems mentioned above): given $\ZZ$-Hodge structures $V_1$ and $V_2$ with trivial determinants
such that $V_1\otimes V_2$ is  a motive, are $V_1$ and $V_2$ motives?

If we do not know the answer to that question, but do assume the Hodge conjecture, then one might be able to show that $V$ of weak geometric origin implies that some $V^{\oplus k}$ is of strong geometric origin.
Some understanding of algebraic cycles and the decomposition theorem would also clearly be needed.

In the context of the present paper, our main constructions provide smooth families of abelian varieties, hence give smooth geometric origin in some cases. One can envision looking at the question of strong geometric origin for those cases---that would require delving into the endomorphism algebras of the families of abelian varieties we construct. We do not make any claims about that here.

\section{Factorization through orbicurves}
\label{sec-fact}

An {\em orbicurve} $C$ is a smooth $1$-dimensional Deligne-Mumford
stack whose generic stabilizer group is trivial.  The coarse moduli
space $C^{\rm coarse}$ is a smooth curve having a collection of points
$p_1,\ldots , p_k$ such that there exist strictly positive integers
$n_1,\ldots , n_k$ such that
$$
C \cong  C^{\rm coarse}\left[ \frac{p_1}{n_1} ,\ldots ,  \frac{p_k}{n_k} \right]
$$
is the root stack. The fundamental group $\pi _1(C,y)$ is the same as
the orbifold fundamental group considered in Subsection
\ref{homotopy-subsection} (with $n_i$ equal to multiplicities of the
fibres over $p_i$).

The goal of this subsection is to prove the equivalence of several different notions of
factorization through orbicurves.

\begin{Definition}
\label{factordef}
Let $X$ be a smooth complex quasi-projective variety.
Suppose that $\rho:\pi_1(X, x)\to \GL (n, \CC)$ is an irreducible representation.
\begin{enumerate}
\item We say that $\rho$ {\em projectively factors through an orbicurve} if there is an orbicurve $C$,
a fibration $f:X\rightarrow C$, and a commutative diagram
$$
\begin{array}{ccc}
\pi _1(X,x) & \rightarrow & \GL (n,\CC ) \\
\downarrow & & \downarrow \\
\pi _1(C,f(x)) & \rightarrow & \PGL (n,\CC ).
\end{array}
$$

\item We say that $\rho$ {\em virtually projectively factors through
    an orbicurve} if there is an alteration $p:Z\rightarrow X$ such
  that $p^{\ast}\rho$ projectively factors through an orbicurve.
\end{enumerate}
\end{Definition}

\medskip

\begin{Definition}
  Let $X$ be a smooth complex quasi-projective variety.  We say that an
  irreducible representation $\rho :\pi_1(X, x)\to \GL (n, \CC)$ is
  {\em tensor decomposable} if the associated local system $V_{\rho}$
  can be written as a tensor product $V_1\otimes V_2$ of two local
  systems of rank $\geq 2$.  Say that $\rho$ is {\em virtually tensor
    decomposable} if there exists an alteration $p:Z\rightarrow X$
  such that $p^{\ast} \rho$ is tensor decomposable.

Say that $\rho$ is {\em virtually reducible} if there is an alteration
$p:Z\rightarrow X$ such that $p^{\ast} \rho$ is reducible, i.e. it
decomposes into a nontrivial direct sum.
\end{Definition}

  Let us note that a representation $\rho :\pi_1(X, x)\to \SL (n,
  \CC)$ is virtually reducible if and only if it is not Lie
  irreducible.  Indeed, if $p:Z\rightarrow X$ is an alteration then the image of
  $\pi_1(Z,z)\to \pi_1(X,x)$ has finite index in
  $\pi_1(X,x)$. So if $\rho$ is virtually reducible then it is not Lie
  irreducible. The implication in the other direction follows from the
  fact that a finite index subgroup in $\pi_1(X,x)$ gives rise to a
  finite \'etale covering.

\medskip

For the proof of the following factorization theorem we refer to Appendix (see Theorem \ref{factorequiv2}):

\begin{Theorem}
\label{factorequiv}
Let $X$ be a smooth complex projective variety. Let us fix an
irreducible representation $\rho:\pi_1(X, x)\to \SL (n, \CC)$ for some
$n\ge 2$. Suppose that $\rho$ is not virtually tensor decomposable,
and not virtually reducible. Then the following conditions are
equivalent.
\begin{enumerate}

\item $\rho$ projectively factors through an orbicurve;

\item $\rho$ virtually projectively factors through an orbicurve;

\item There exists a map $f: X\rightarrow C$ to an orbicurve and a
  fiber $F=f^{-1}(y)$, such that the restriction of $\rho$ to $\pi
  _1(F,x)$ becomes reducible;

\item There exists an alteration $p:Z\rightarrow X$ such that the
  previous condition holds for the pullback $p^{\ast}\rho$.

\end{enumerate}
\end{Theorem}

\medskip

\begin{Lemma} \label{virtually-reducible-is-not-(1,1,1)}
  Suppose $\rho : \pi_1(X, x)\to \SL (n, \CC)$ is an irreducible representation of rank $n\leq
  3$. Then $\rho$ is not virtually tensor decomposable.  If $\rho$ is
  virtually reducible then it has image either in a finite subgroup,
  or in the normalizer of a maximal torus. If $\rho$ is virtually
  reducible and underlies a VHS then it is unitary with a single Hodge
  type.
\end{Lemma}

\begin{proof}
  Clearly $\rho$ cannot be virtually tensor decomposable because then
  it should have rank $\geq 4$. Suppose $\rho$ is virtually reducible,
  becoming reducible upon pullback to $p:Z\rightarrow X$. The image of
  $\pi _1(Z,z)$ is of finite index in $\pi _1(X,x)$ and the
  restriction of $\rho$ to this subgroup is reducible. Therefore we
  may replace $Z$ by the finite \'etale covering corresponding to this
  subgroup, in other words we may suppose that $p$ is finite
  \'etale. We may furthermore suppose it is Galois with group $G$.
  Write the decomposition into isotypical components
$$
V_{\rho}|_Z = \bigoplus V_i\otimes W_i,
$$
where $V_i$ are irreducible local systems on $Z$ and $W_i$ are vector
spaces. The Galois group acts on $\{ V_i\}$. This action is
transitive, otherwise $\rho$ would be reducible. In particular, all of
the $V_i$ and all of the $W_i$ have the same rank.

Suppose there is more than one isotypical component. Then $V_i$ and
$W_i$ must be of rank $1$ and $G$ acts on the set of isotypical
  components by permutation. Since the subgroup of $\SL(n, \CC)$
  fixing the decomposition of $V_{\rho}|_Z $ into isotypical
  components is a maximal torus, the image of $\rho$ is contained in
  the normalizer of a maximal torus.  Furthermore, assume that $\rho$
underlies a VHS. The Galois action preserves the Hodge type, and there
is a single Hodge type for the rank $1$ isotypical component (note
that the isotypical decomposition is compatible with the Hodge
structure). Therefore $\rho$ has only a single Hodge type and it is
unitary. This proves the lemma in the case of several isotypical
components.

We may now assume $V_{\rho}|_Z = V_1\otimes W_1$. Since the rank is
$\leq 3$, and by hypothesis $V_{\rho}|_Z$ is reducible so the rank of
$V_1$ is strictly smaller than $n$, we get that $V_1$ has rank $1$. It
means that the restriction and projection to a representation $\pi
_1(Z,z)\rightarrow \PGL (n,\CC )$ is trivial. Since by hypothesis our
representation is into $\SL (n,\CC )$ we get that $\rho$ has finite
image. In particular it is unitary, and a unitary irreducible VHS can
have only a single Hodge type.  This completes the proof of the lemma.
\end{proof}

\medskip

\begin{Corollary} \label{virtually-implies-usually} Let $X$ be a
  smooth complex projective variety, and suppose $\rho:\pi_1(X,
  x)\to \SL (3, \CC)$ is an irreducible representation of rank $3$
  underlying a VHS with more than a single Hodge type. Then the
  conditions listed in Theorem \ref{factorequiv} are equivalent.
\end{Corollary}
\begin{proof}
  By the above lemma, $\rho$ is not tensor decomposable and not virtually
  reducible.
\end{proof}

\section{Structure of complex VHS of type $(1,1,1)$}\label{structure=(1,1,1)}

In this section we study general complex VHS of type $(1,1,1)$ on
quasi-projective varieties.  The main aim is to prove that every
complex VHS of type $(1,1,1)$ or its dual can be constructed
analogously to the following example:

\begin{Example}\label{general-form}
Let $X$ be a smooth complex projective surface. Let $L_1$ be a nef
line bundle with strictly positive degree, $L_1^2=0$ and fixed
inclusion $j_1:L_1\hookrightarrow \Omega_X$.  Let us fix a line bundle
$j_2:L_2\hookrightarrow \Omega_X$
generically the same as $j_1(L_1)$
such that $c_1(L_2)=a\, c_1(L_1)$ in $H^2(X, \QQ)$ for some
(rational) $a\in (-\frac{1}{2}, 1]$.  Finally, let us fix a line bundle
  $L_0$ such that $3c_1(L_0)=(2+a)c_1(L_1)$ in $H^2(X,\QQ)$.

Let us consider the system of Hodge bundles defined by
$$E^{2,0}:=L_0, \quad E^{1,1}:=L_0\otimes L_1^{-1}, \quad
E^{0,2}:=L_0\otimes L_1^{-1}\otimes L_2^{-1}$$ with $\theta$ given by
inclusions $j_i:L_i\hookrightarrow \Omega_X$ tensored with identity on
$E^{p,q}$ for $(p,q)=(2,0), (1,1)$. Then for any ample line bundle $A$ the pair $(E, \theta)$ is an
$A$-stable system of Hodge bundles with vanishing rational Chern
classes.

Later we will see that if $a\ne 1$ then this system of Hodge bundles comes from an orbicurve
(see Theorem \ref{main_thm_on_type_(1,1,1)}).
\end{Example}

\medskip

Let $X$ be a smooth complex quasi-projective variety of dimension
$d\ge 2$.  Let us fix a smooth projective variety $\bar X$ containing
$X$ as an open subset and such that $D=\bar X-X$ is a simple normal
crossing divisor.
In this section, we consider the quasi-projective case for future reference, although afterwards in the present paper we shall assume
$D=\emptyset$.

Let $A$ be a fixed very ample divisor on $\bar
X$. In the following we use the intersection pairing of $\QQ$-divisors
on the polarized variety $(\bar X, A)$ as defined in Subsection
\ref{intersection-pairing}.

\begin{Remark}
Technically speaking, one should make a distinction
between a line bundle and its corresponding divisor class.
In the subsequent numerical calculations that would
add burdensome extra notation, so we make the convention
that a symbol such as, typically, $L_i$ can either mean
the line bundle or the corresponding divisor class, according
to context.
In particular, we usually use the tensor product $L_1\otimes L_2$ but in computation of
intersection numbers we use divisor classes and write, e.g.,  $(L_1+2L_2).A$ instead of $(L_1\otimes L_2^{\otimes 2}).A$.

\end{Remark}

Let us recall that by Bertini's theorem the general complete
intersection surface $\bar Y$ in $\bar X$ is smooth and
irreducible. We can also assume that $D_Y=\bar Y\cap D$ is a simple
normal crossing divisor on $\bar Y$. Let us set $Y:=\bar Y-D_Y$.

Let $(E=\bigoplus_{p+q=2}E^{p,q}, \theta)$ be a rank $3$
logarithmic system of Hodge bundles of type $(1,1,1)$. Recall
that this definition includes the condition that both
Kodaira-Spencer maps $\theta$ be nonzero, so they induce
injections of rank $1$ sheaves
$$L_1:=E^{2,0}\otimes (E^{1,1})^*\hookrightarrow
\Omega _{\bar X}(\log D)$$ and
$$L_2:=E^{1,1}\otimes (E^{0,2})^*\hookrightarrow
\Omega _{\bar X} (\log D).$$

The proof of the following proposition is inspired by the proof of
\cite[Theorem 4.2]{La1}.

\begin{Proposition} \label{rank-3-structure} Assume that $E$ has
  vanishing rational Chern classes and that $(E, \theta)$ is slope
  $A$-stable.  Then the following conditions are satisfied:
    \begin{enumerate}
    \item  $c_1(E^{p,q}).c_1(E^{p',q'})=0$ for all pairs $(p,q)$ and
  $(p',q')$.
 \item The classes of $L_1$ and $L_2$ lie on the same line in the vector space
    $H^2(\bar X, \QQ)$.
  \item There exist effective divisors $B_1$ and $B_2$ such that
    $M:=L_1(B_1)=L_2(B_2)$ is a saturation of both $L_1$ and $L_2$ in
    $\Omega _{\bar X}(\log D)$. Moreover,  we have $L_i.M=0$ and $L_i.B_j=0$ for all $i$ and $j$.
    \item Either $L_1|_{\bar Y}$ or $L_2|_{\bar Y}$ is nef with self intersection zero and strictly positive degree.
    \end{enumerate}
\end{Proposition}

\begin{proof} We claim that the map
$$\alpha : L_1\oplus L_2\to \Omega_{\bar X} (\log D)$$
induced from $\theta$ has rank $1$. Indeed, in the local coordinates
$\alpha$ is given by two logarithmic $1$-forms $\omega_1$ and $\omega_2$ and the
integrability of $\theta$ implies that $\omega_1\wedge \omega_2=0$, so
they are proportional and $\alpha$ has rank $1$. Let $M$ be the
saturation of the image of $\alpha$ in $\Omega_{\bar X}(\log D)$ (i.e., the largest
rank $1$ subsheaf of $\Omega_{\bar X}(\log D)$ containing the image of $\alpha$). Then $M$
is a line bundle (as it is a rank $1$ reflexive $\cO_X$-module on a smooth variety) such that $\Omega_{\bar X}(\log D)/M$ is torsion free.

By construction we have inclusions
$L_1 \subset M$
and
$L_2 \subset M$
defining effective divisors $B_1$ and $B_2$, respectively. Let us set
$e^{p,q}=c_1(E^{p,q})$, $l_i=c_1(L_i)$ and $b_i=c_1(B_i)$. Then we can
write
$$m=c_1(M)=l_1+b_1=l_2+b_2.$$
By assumption we have
$$e^{0,2}+e^{1,1}+e^{2,0}= 0$$
in the rational cohomology $H^2(\bar X, \QQ)$, so the classes $e^{p,q}$ can
be written in terms of $l_1$ and $l_2$ in the following way:
$$
e^{2,0}= \frac{1}{3} (2l_1+ l_2), \quad e^{1,1} =\frac{1}{3} (l_2-l_1) ,\quad
e^{0,2}=\frac{1}{3} (-l_1-2l_2).
$$
Therefore
$$3 m=l_1+b_1+2(l_2+b_2)\ge
l_1+2l_2=-3e^{0,2}.$$
In particular, we have $m\ge -e^{0,2}$.
Similarly, we have
$$3 m=2(l_1+b_1)+(l_2+b_2)\ge 2l_1+l_2=3e^{2,0},$$
so $m\ge e^{2,0}$.

 The following lemma is a corollary of
the well-known Bogomolov's lemma:

\begin{Lemma}\label{not-big}
For any line bundle $N\subset \Omega_{\bar X}(\log D)$ the restriction $N|_{\bar Y}$ is not big on $\bar Y$.
\end{Lemma}

\begin{proof} We have a short exact sequence
$$0\to \cO_{\bar Y} (-A)^{n-2}\to \Omega_{\bar X}(\log D)|_{\bar Y}\to \Omega_{\bar Y}(\log D_Y)\to 0.$$
Let us consider the composition $N|_{\bar Y}\to \Omega_{\bar X}(\log D)|_{\bar Y}\to \Omega_{\bar Y}(\log D_Y)$.
By Bogomolov's lemma (see \cite[Theorem 4]{Bog} and \cite[Corollary 6.9]{EV})  $\Omega_{\bar Y}(\log D_Y)$ does not contain big line bundles.
So if this composition is non-zero then $N|_{\bar Y}$ is not big.

If the map $N|_{\bar Y}\to \Omega_{\bar Y}(\log D_Y)$ is zero then by the above exact sequence
$N|_{\bar Y}$ is contained in $\cO_{\bar Y} (-A)^{n-2}$, so $\kappa (N|_{\bar Y})=-\infty$.
\end{proof}

\medskip

Slope $A$-stability of  $(E, \theta)$ implies that
$$e^{0,2}. \, A < 0$$
and
$$-e^{2,0}. \, A=(e^{0,2}+e^{1,1}). \, A < 0.$$
If $(e^{0,2})^2>0$ then,  thanks to Lemma \ref{nefness}, the first inequality implies that $-e^{0,2}|_{\bar Y}$
is big. But then $M|_{\bar Y}$ is big, which contradicts the above lemma.
Therefore $$(e^{0,2})^2\le 0.$$ Similarly, if $(e^{2,0})^2>0$ then
$e^{2,0}|_{\bar Y}$ is big. But then $M|_{\bar Y}$ is big, a contradiction. Therefore
$$(e^{2,0})^2\le 0.$$

Let us also recall that $$0=c_2(E)=e^{2,0}.\, e^{1,1}+e^{1,1}.\,
e^{0,2}+e^{0,2}. \, e^{2,0}=-(e^{2,0})^2+e^{1,1}.\, e^{0,2}$$
Therefore
$$L_2^2=(e^{1,1}+e^{0,2})^2-4e^{1,1}.\, e^{0,2}=(e^{2,0})^2-4(e^{2,0})^2=-3(e^{2,0})^2\ge 0.$$
Similarly, we have
$$L_1^2=(e^{2,0}-e^{1,1})^2=-3 (e^{0,2})^2\ge 0.$$

From stability of $E$ we get $(L_1+2L_2).A=-3e^{0,2}.A>0$ and hence there exists
$i$ with $L_i.A>0$. Lemma \ref{nefness} and Lemma \ref{not-big} imply 
that $L_i|_{\bar Y}$ is nef
with $L_i^2=0$ and the degree of $L_i|_{\bar Y}$ is strictly positive.

\begin{Lemma}\label{intersection}
  We have $L_1^2=L_1.L_2=L_2^2=0$.
\end{Lemma}

\begin{proof}
  Let us first assume that $i=1$, i.e., $L_1|_{\bar Y}$ is nef with
  $L_1^2=0$. Since $B_2$ is effective, we have $L_1.B_2\ge 0$.  By
  Lefschetz' hyperplane theorem for quasi-projective varieties (see
  \cite[Theorem 1.1.3]{HL}) we have $\pi_1(Y)\simeq \pi_1(X)$. So by
  functoriality of the correspondence between representations and
  Higgs bundles, the restriction
$$(E|_{\bar Y},\theta_Y: E|_{\bar Y}\mathop{\to}^{\theta|_Y} E|_{\bar Y}\otimes \Omega_{\bar X}(\log D)|_{\bar Y}\to E|_{\bar Y}\otimes \Omega_{\bar Y}(\log D_Y))$$
of $(E,\theta)$ to $\bar Y$ is $A$-stable (one can also give a direct
proof of this fact: see \cite[Theorem 12]{La2} and \cite{La3}).  Since
the rational Chern classes of $E|_{\bar Y}$ vanish, $(E|_{\bar
  Y},\theta_Y)$ is stable with respect to every stable polarization
and semistable with respect to every nef polarization.  In particular,
$L_1|_{\bar Y}$-semistability of $(E|_{\bar Y}, \theta _Y)$ implies
that $L_1 . (L_1+2L_2) \ge 0$. Therefore we have $L_1.L_2\ge 0$.

  Applying Lemma \ref{Zariski-semi-decomposition} to $M=L_1+B_1$ we
  see that $L_1^2=0$, $L_1.B_1=0$ and $L_1.M=0$. Using $M=L_2+B_2$ we get equalities
  $L_1.L_2=L_1.B_2=0$. But then
$$c_2(E)=\frac{1}{9}(L_1^2+L_1.L_2+L_2^2)=0$$ implies that
  $L_2^2=-L_1.L_2=0$.

  The proof in the case $i=2$ is analogous.
\end{proof}

This lemma finishes the proof of assertions (1) and (4) of the
proposition. To prove (2), let us consider $i'$ such that
$\{i,i'\}=\{1,2\}$. Let us choose a rational $a$ such that
$(L_{i'}-aL_i).A=0$. Since $(L_{i'}-aL_{i})^2=0$ this implies that
$l_{i'}=a l_{i}$ in $H^2(X, \QQ)$.

Assertion (3) of the proposition follows from the proof of Lemma
\ref{intersection}.  More precisely, if $i=1$ then this proof shows
that $L_1.M=L_1.B_1=L_1.B_2=0$. So (2) implies that
$L_2.M=L_2.B_1=L_2.B_2=0$. The proof in case $i=2$ is analogous.
\end{proof}

\begin{Remark}
Let us remark that stability of $E$ implies that $(L_1+2L_2).A>0$ and
$(2L_1+L_2).A>0$ so that $a>-\frac{1}{2}$ in the notation of the above
proof.  In fact, this condition is sufficient to define a stable rank
$3$ logarithmic system of Hodge bundles.
\end{Remark}

\medskip

{
Let us recall that a smooth log pair is a pair consisting of a smooth variety and a (reduced)
simple normal crossing divisor. A morphism of log pairs $\bar f: (\bar
Z, D_Z)\to (\bar X, D)$ is a proper morphism of normal varieties such that
$\bar f (D_Z)=D$.}
For a morphism of smooth projective log pairs $\bar f: (\bar
Z, D_Z)\to (\bar X, D)$, we write $Z:=\bar Z-D_Z$, we choose a point $z$
over $x$ and we set $f=\bar f|_{\bar Z-D_Z}$.

\medskip

Let $\rho: \pi_1(X, x)\to \SL (n, \CC)$ be an irreducible
representation with quasi-unipo\-tent monodromy at infinity (i.e., such
that $\rho$ has quasi-unipotent monodromies along small simple loops
around all irreducible components of $D$). Then by Kawamata's covering
trick there exists a finite flat morphism of smooth projective log
pairs $\bar f: (\bar Z, D_Z)\to (\bar X, D)$ such that $({\bar
  f}^*D)_{red}=D_Z$ and $f^*\rho: \pi_1(Z,z)\to \SL (n, \CC)$ has
unipotent monodromy at infinity, i.e., $f^*\rho$ has unipotent local
monodromies along all irreducible components of $D_Z$. This implies
that the residues of Deligne's canonical extension of the flat bundle
associated to $f^*\rho$ are nilpotent (since all the eigenvalues of
the residues are zero) and we are interested in bundles with trivial
parabolic structure.

More precisely, let us consider a vector bundle with connection $(V,
\nabla)$ with regular singularities and rational residues in $[0,1)$
along the irreducible components of $D$, that corresponds to the
representation $\rho$. This bundle comes equipped with a canonical
parabolic structure, which makes the corresponding parabolic flat
bundle stable with vanishing parabolic Chern classes (cf. \cite{Mo1},
\cite{Mo2} and \cite[Lemma 3.3]{IS1}).

If we assume the local system $V_{\rho}$ underlies a complex variation
of Hodge structure then $f^*\rho$ also underlies a complex VHS.  The
logarithmic parabolic Higgs bundle corresponding to $\rho$ is stable
and it has vanishing parabolic first and second Chern classes (see
\cite{Mo1} and \cite{Mo2}). The pullback of this logarithmic parabolic
Higgs bundle to $(\bar Z, D_Z)$ (cf. \cite[Lemma 3.7]{IS1}) becomes a
stable logarithmic system of Hodge bundles (with trivial parabolic
structure) which has vanishing all rational Chern classes (by
functoriality of the Kobayashi--Hitchin correspondence, the Higgs
bundle associated to the pull back $f^*\rho$ corresponds to a complex
VHS, so it is a system of Hodge bundles).  Thus we have the following
lemma:

\begin{Lemma}
  Let $\rho: \pi_1(X, x)\to \SL (n, \CC)$ be an irreducible
  representation with quasi-unipotent monodromy at infinity. If $\rho$
  underlies a complex VHS then there exists a finite flat morphism of
  smooth projective log pairs $\bar f: (\bar Z, D_Z)\to (\bar X, D)$
  such that $f^*\rho: \pi_1(Z,z)\to \SL (n, \CC)$ has unipotent
  monodromy at infinity.  To this representation the
  Kobayashi--Hitchin correspondence associates a stable rank $n$
  system of logarithmic Hodge sheaves with vanishing rational Chern
  clases.
\end{Lemma}

In case the underlying rank $3$ logarithmic system of Hodge bundles on
$(\bar Z, D_Z)$ is of type $(1,1,1)$, we can apply Proposition
\ref{rank-3-structure}.

\medskip

\begin{Remark}
  If $\rho$ is a rigid irreducible representation with
  quasi-unipotent monodromy at infinity, then it underlies a complex
  VHS (this part of \cite[Theorem 8.1]{CS} works in any rank) and we
  can apply the above lemma to such representations.
\end{Remark}

\section{Geometry of complex VHS of type $(1,1,1)$}
\label{sec-geometry}

Let $X$ be a smooth complex projective variety. If $\rho:
\pi_1(X,x)\to\SL (n,\CC)$ is a representation then the induced
projective representation $ \pi_1(X)\to \PGL (n,\CC)$ is denoted by
$\bar \rho$.

\begin{Theorem}\label{main_thm_on_type_(1,1,1)}
  Let $\rho:  \pi_1(X,x)\to \SL (3,\CC)$ be an irreducible representation coming
  from a complex variation of Hodge structure of type $(1,1,1)$.  Let
  $V_{\rho}$ denote the corresponding local system.  Then one of the
  following holds:
\begin{enumerate}
\item There exists a projective representation $\pi_1(X,x)\to
    \PGL (2,\CC)$ which induces $\bar\rho$ via the homomorphism $\PGL(2,
    \CC)\to \PGL (3, \CC)$ given by second symmetric power. Moreover, there exists a finite covering
    $\pi:Y\to X$ from a smooth projective variety $Y$, a rank $1$
    local system $W_1$ on $X$ such that $W_1^{\otimes 3}$ is trivial
    and a rank $2$ local system $W_2$ on $Y$ such that
    $\pi^*V_{\rho}=\pi^*W_1\otimes \Sym ^2W_2$.
\item $\rho$ projectively factors through an orbicurve.
\end{enumerate}
\end{Theorem}

\begin{proof} We can assume that $X$ has dimension at least $2$ as
  otherwise there is nothing to be proven.  Therefore the complex
  variation of Hodge structure corresponding to $\rho$ is of the form
  described in Proposition \ref{rank-3-structure}. Let us write
  $L_2\equiv aL_1$. Since the assertion for a representation is
  equivalent to the assertion for its dual, we can assume that
 for a general complete intersection surface $S\subset X$ the restriction $L_1|_S$ is nef with strictly positive degree and $a\le 1$.

  In the notation of Proposition \ref{rank-3-structure} we have
  $M= L_1(B_1)= L_2(B_2)$.  Let us write $B_1=B+B_1'$ and
  $B_2=B+B_2'$, where $B$, $B_1'$, $B_2'$ are effective and $B_1'$ and
  $B_2'$ have no common irreducible components. Recall that
  $L_1.B=L_1.B_1'=0$ and $L_2.B=L_2.B_2'=0$. We set $M':=M(-B)$. This
  line bundle has a canonical inclusion into $\Omega_X$ given by
  composing $M(-B)\subset M$ with $M\to \Omega_X$. By definition we
  also have $M'=L_1(B_1')=L_2(B_2')$.  Intersecting both sides of
  equality $L_1+B_1'=L_2+B_2'$ with $B_1'$ we get
$$(B_1')^2=B_1'.B_2'\ge 0$$
(here we use that $L_2.B_1'=aL_1.B_1'=0$,
and since
$B_1'$ and $B_2'$ have no common components their
intersection is positive).
But Lemma \ref{Zariski-semi-decomposition} implies that 
$(B_1')^2\le 0$, so we have $(B_1')^2=B_1'.B_2'=0$. Hence we also
get $(B_2')^2=0$.  There exists also some non-negative $\alpha_1$ such
that $(B_1'-\alpha_1L_1).A=0$. Since $(B_1'-\alpha_1L_1)^2=0$, we have
$B_1'\equiv \alpha_1L_1$. Similarly, there exists some non-negative
$\alpha_2$ such that $B_2'\equiv \alpha_2L_1$.  Note also that
$B_i'=0$ is equivalent to $\alpha_i=0$, since if $B_i'$ is a non-zero effective
divisor then $B_i'.A>0$.  Let us recall that
$L_1+B_1'=L_2+B_2'$, so $(1+\alpha_1)L_1\equiv (a+\alpha_2)L_1$.  But
$L_1.A>0$, so $1+\alpha_1=a+\alpha_2$. Since $a\le 1$, this implies
that $0\le \alpha_1\le \alpha_2$.

We consider two cases depending on whether $B_2'$ is zero or not.

If $B_2'=0$ then $ \alpha_2=\alpha_1=0$, $B_1'=0$ and $L:=L_1=L_2$
with the same map to $M$ and $\Omega_X$.  In this case we set
$N:=E^{2,0}\otimes L_1^{-1}$.  Since $\det E=N^{\otimes 3}$, the line
bundle $N$ is $3$-torsion.  Therefore $(E, \theta)\simeq (N,
  0)\otimes \Sym^2(F)\otimes \det (F)^{-1}$, where $F=(F^{1,0}\oplus
F^{0,1}, \theta_F)$ is a system of Hodge bundles with $F^{1,0}=L$,
$F^{0,1}=\cO_X$ and the Higgs field $\theta_F$ given by the inclusion
$L\to \Omega_X$.  Note that $F$ is an $A$-stable system of Hodge
bundles but it does not have vanishing Chern classes so it does not
underlie a rank $2$ representation of $\pi_1(X,x)$. But it comes from
a projective representation $\pi_1(X,x)\to \PGL(2,\CC)$.  There exists a
finite covering $\pi: Y\to X$ such that $\pi^*L=M^{\otimes 2}$ for
some line bundle $M$ on $Y$. Then $\pi^*(E, \theta)\simeq \pi^* (N,
0)\otimes \Sym^2(F_Y)$, where $F_Y=(F_Y^{1,0}\oplus F_Y^{1,0},
\theta_{Y})$ is a system of Hodge bundles with $F^{1,0}_Y=M$,
$F^{0,1}_Y=M^{-1}$ and the Higgs field $\theta_{Y}$ given by the
inclusion $\pi^*L\to \pi^*\Omega_X\to \Omega_Y$.

This corresponds to case (1) of the theorem.

If $B_2'\ne 0$,
then by the Hodge index theorem there exist positive integers $b$ and $c$ such
that $bM'\equiv cB_2'$.
The idea is now to say that $M'$ looks like an effective divisor. Roughly speaking, we try to get a section of $M'$ and then argue using this differential form. In practice we need to extract a root, and we only have numerical information so, in the case where the Picard scheme has positive dimension, we need to deal with several cases.

Start by noting that
$U:=bM'-cB_2'\in \Pic ^{\tau}
(X)$. Multiplying $b$ and $c$ by the same positive integer we can
assume that in fact $U\in \Pic ^0(X)$. Therefore we have a nonzero map
$$
U \rightarrow (M')^{\otimes b}
$$
determined by the divisor $cB_2'$. We can write $U=V^{\otimes b}$
for some $V\in \Pic^0(X)$ and hence we get a section $\eta \in
H^0(X,(M'\otimes V^{\ast})^{\otimes b})$ (whose divisor is $cB_2'$).

Let $p:Z\rightarrow X$ be a desingularization of the ramified covering
defined by taking the $b$-th root of $\eta$.  Over $Z$ we have a
tautological section $\alpha \in H^0(Z, p^{\ast}(M'\otimes V^{\ast}))$
such that $\alpha ^{\otimes b}=p^*\eta$.

In particular, $\alpha \in H^0(Z, \Omega ^1_Z\otimes
p^{\ast}V^{\ast})$, which means that $p^{\ast}V^{\ast}\in \Pic ^0(Z)$
is in the jump-locus for twisted sections of $\Omega ^1_Z$:
$$S(Z):=\{N\in \Pic^0(Z)|\, \dim H^0(Z, \Omega_Z^1\otimes N)\ge 1\}.$$
The irreducible components of this jump-locus are translates of
abelian subvarieties of ${\rm Pic}^0(Z)$ by torsion points.
The fact that the components are translates of abelian varieties
was proved in \cite{GL} and the assertion about torsion was proved in \cite{Si-jump}.
We distinguish two cases depending on whether $p^{\ast}V^{\ast}$ is a
torsion point or not.

\medskip

\emph{Case 1.} Let us assume that $p^{\ast}V^{\ast}$ is a torsion point of
${\rm Pic}^0(Z)$.

\medskip

Then the line bundle $p^{\ast}V^{\ast}$ defines a
finite \'etale covering $q:Z'\rightarrow Z$ such that $V^{\ast}$
becomes trivial after pulling back to $Z'$. Let $p'$ denote the
composition of $q:Z'\to Z$ with $p:Z\to X$.  Then over $Z'$ we get a nonzero
section $\xi\in H^0(Z', \Omega ^1_{Z'})$ given by the composition of
$q^*\alpha: \cO_{Z'}\to q^*(p^{\ast}(M'\otimes V^{\ast}))\simeq
q^*(p^{\ast}(M'))=(p')^*M'$ with the canonical map $(p')^*M'\to
(p')^*\Omega_X\to \Omega^1_{Z'}$.

The line bundle $(p')^{\ast}M'\subset
\Omega^1_{Z'}$ is generically generated by this section $\xi$. 

We can look at the Albanese map defined by $\xi$ (see
\cite[p.~101]{Si2}). This is the map $\psi: Z'\to A=\Alb (Z')/B$,
where $B$ is the sum of all abelian subvarieties of $\Alb (Z')$ on
which $\xi$ vanishes. By construction $\xi=\psi^*(\xi_A)$ for some
$1$-form $\xi_A$ on $A$ such that the restriction of $\xi_A$ to any
nontrivial abelian subvariety of $A$ is nonzero.

We distinguish two subcases depending on the dimension of $\psi (Z')$:

\medskip

\emph{Subcase 1.1.} The image of $\psi$ is a curve.

\medskip

In this case taking the Stein
factorization of $Z'\to \psi (Z')$, we get a fibration $f:Z'\to C$
over a smooth projective curve and a $1$-form $\xi_C$ on $C$ such
that $\xi=f^*\xi _C$. Considering $\xi$ as a map $\cO_{Z'}\to
\Omega_{Z'}$, this means that the map factors through $f^*\xi_C:
\cO_{Z'}=f^*\cO_C\to f^*\Omega_C$. Therefore $(p')^*M$ is contained
in the saturation of $f^*\Omega_C$ in $\Omega_{Z'}$.  If $F$ is a
smooth fiber of $f$ (or just a multiplicity $1$ irreducible
component of a fiber of $f$) then this shows that the canonical
map $(p')^*M|_F\to \Omega_{Z'/C}|_F=\Omega_F$ is zero.
This implies that the Kodaira-Spencer maps of our variation of
Hodge structures restricted to the fiber $F$, vanish.
Each Hodge subbundle is therefore a flat subbundle.
Hence the
restriction of our variation of Hodge structures to $F$ splits into
a direct sum of three rank $1$ variations of Hodge structure
at the three different Hodge types.
Clearly, $\rho$ is not virtually tensor decomposable. It is also  not virtually reducible as
the Kodaira-Spencer maps remain
nonzero under alterations.
So by Theorem \ref{factorequiv} ($4\Rightarrow 1$), $\rho$ projectively factors through an
orbicurve.

\medskip

\emph{Subcase 1.2.} The dimension of the image  $\psi (Z')$ is at least $2$.

\medskip

Let $\tilde A\to A$ be the universal covering of $A$ (treated as a
complex vector space). This map is only analytic and not algebraic.
Let $\tilde Z:=Z'\times_A\tilde A$ be the covering of $Z'$ defined by
the Albanese map $\psi$ and let $\pi: \tilde Z\to Z'$ be the
projection on the first factor and $\tilde \psi: \tilde Z\to \tilde A$
on the second factor (these maps are also only analytic). Let us fix a
point $\tilde z_0\in \tilde Z$.  There exists a unique linear function
$g_{\tilde A}:{\tilde A}\to \CC$ such that $g_{\tilde A}(\tilde \psi
(\tilde z_0))=0$ and $dg_{\tilde A}=\xi _A$. Let $g: \tilde Z\to \CC$
be the composition of $g_{\tilde A}$ with $\tilde \psi$.  Then the
Lefschetz theorem (see \cite[Theorem 1]{Si2}; since $\dim (\psi (Z'))\ge 2$ the assumptions of
  this theorem are satisfied) says that
for any fibre $F$ of $g$ (or in other words a possibly
  non-compact leaf of the foliation defined by $\xi$), the map $\pi
_1(F, \tilde z_0) \rightarrow \pi _1(\tilde{Z}, \tilde z_0)$ is
surjective.

Note that by our construction the restriction $\pi^*\xi |_F$ is zero.
Choosing $F$ so that it does not lie in the preimage of $B_2'$, the
inclusion $M'\to \Omega_X$ becomes zero after pulling back to $F$ and
composing with the canonical map to $\Omega_F$. So our variation of
Hodge structure splits on $F$ into a direct sum of three line bundles.
In this case the image of the representation $\pi_1(F,\tilde z_0)\to \SL (3,
\CC)$ is contained in a maximal torus.  Therefore the monodromy
representation is abelian on $\pi _1(\tilde{Z}, \tilde z_0)$.  But $\pi _1(Z', \tilde \psi (\tilde z_0))/\pi
_1(\tilde{Z}, \tilde z_0)$ is an abelian group, so the monodromy of our original
VHS would be solvable. Then the Zariski closure $G\subset \SL(3, \CC)$ of
  $\mathop{\rm im}\, \rho$ is also solvable (this follows from
  \cite[I.2.4 Proposition]{Bo}).  On the other hand, since $\rho$ is
  irreducible, the connected component $G^{0}$ is reductive. Hence
  $G^{0}$ is contained in a maximal torus of $\SL (3, \CC)$.
Taking the \'etale covering $q:\tilde X \to X$ corresponding to the image of
$\pi_1(X,x)\to G/G^{0}$, we see that the monodromy of $q^*V_{\rho}$ is contained
in the maximal torus of $\SL (3,\CC)$. But then the corresponding pull-back of
the system of Hodge bundles is a direct sum of line bundles with trivial maps between
the components, contradicting the fact that the original VHS was of type $(1,1,1)$.

\medskip

\emph{Case 2.} Assume that $p^{\ast}V^{\ast}$ is not a torsion point.

\medskip

Then an irreducible component of $S(Z)$ containing this point has dimension
$\ge 1$.

Let us recall that for a line bundle $N\in \Pic^0(Z)$ we have an isomorphism
of complex vector spaces
$$H^0(Z, \Omega_Z^1\otimes N)\simeq \overline{H^1(Z, N^{-1})},$$
so
$$S(Z)=\{N\in \Pic^0(Z)|\, \dim H^1(Z, N^{-1})\ge 1\}.$$
Beauville proves in \cite{Be} that the positive-dimensional components of this set all
come from maps to orbicurves. More precisely, \cite[Corollaire 2.3]{Be} says that
there exists a fibration $f: Z\to C$ over a smooth projective curve of genus $g\ge 1$
such that $p^{\ast}V^{\ast}$ is trivial on every smooth fiber of $f$.

\medskip
The proof of the following lemma is modelled on the proof of
\cite[Proposition 1.10]{Be}.

\begin{Lemma}
  Let $Z$ be a smooth complex projective variety and let $C$ be a
  smooth complex projective curve.  Let $f:Z\to C$ be a fibration with
  reduced fibers. Let $L\in \Pic ^{\tau} Z$ be a non-torsion line
  bundle, which is trivial on some smooth fiber of $f$. If $0\ne
  \alpha\in H^0(\Omega_Z\otimes L^{-1})$ then there exists a line
  bundle $L_C$ on $C$ and $\beta \in H^0(\Omega_C\otimes
  L_C^{-1})=\Hom (L_C, \Omega_C)$ such that $L=f^*L_C$ and the
  composition
$$L=f^*L_C\mathop{\longrightarrow}^{f^*\beta}f^*\Omega_C\longrightarrow \Omega_Z$$
corresponds to $\alpha$.
\end{Lemma}

\begin{proof}
  First let us recall that existence of $L_C$ such that $L=f^*L_C$
  follows from \cite[Proposition 1.2]{Be}.  We have a short exact
  sequence
$$0\to H^1(C, L_C)\to H^1(Z, L)\to H^0(C, R^1f_*L)\to 0$$
coming from the Leray spectral sequence. Using the isomorphisms
$H^1(Z, L)\simeq \overline{H^0(Z, \Omega_Z^1\otimes L^{-1})}$ and
$ H^1(C, L_C) \simeq \overline{H^0(C, \Omega_C^1\otimes
L_C^{-1})}$,  it is sufficient to show that $ H^0(C, R^1f_*L)=0$.
But $ H^0(C, R^1f_*L)=H^0(C, L_C\otimes R^1f_*\cO_Z)$ and
$R^1f_*\cO_Z$ is dual to $f_*\omega_{Z/C}$, so we need to check
that $\Hom (f_*\omega_{Z/C}, L_C)=0$.  Let us recall that by
\cite[Theorem]{Fu} (see also \cite[Theorem 17]{CD}),
$f_*\omega_{Z/C}$ is a direct sum of an ample vector bundle and a
direct sum of stable vector bundles of degree zero. Therefore
$\Hom (f_*\omega_{Z/C}, L_C)\ne 0$ if and only if $L_C$ is
isomorphic to one of the direct summands of $f_*\omega_{Z/C}$. But
by Deligne's result, the rank $1$ summands of $f_*\omega_{Z/C}$
are torsion (see \cite[Corollary 21]{CD}) and $L_C$ is not
torsion, so we get a contradiction.
\end{proof}

\medskip

Taking a semistable reduction of $f$ we can assume that all the fibres
of $f$ are reduced. Then, by the above lemma, there exists a line bundle $L_C\in \Pic^0(C)$
such that $p^{\ast}V=f^*L_C$ and the section $\alpha$ comes from a
section of $H^0(C,\Omega_C^1\otimes L_C^{\ast})$.  This shows that
$p^{\ast}V^{\ast}\to f^*\Omega_C$ is an isomorphism at the generic
point of $Z$.  Since $p^{\ast}V\to p^{\ast}M$ is also an isomorphism
at the generic point of $Z$, the sheaf $p^*M$ is contained in
$f^*\Omega_C\subset \Omega_{Z}$ (since the fibres of $f$ are reduced,
$f^*\Omega_C$ is saturated in $\Omega_{Z}$).

As before this implies that the pull back of our variation of Hodge
structure to a fibre of $f$ splits into a direct sum of three rank $1$
variations of Hodge structure. Therefore Corollary
\ref{virtually-implies-usually} implies that $\rho$ projectively
factors through an orbicurve.
Alternatively,  the same arguments as before show that
$((p')^*M')^{\otimes b}$ is a line bundle associated to a certain sum
of rational multiples of fibers of $f$ (but not necessarily positive).
Then by passing to a certain cover $Z''\to Z'$ we conclude that the
pull back of our original variation of Hodge structure to $Z''$ is
isomorphic to the pull back of a variation of Hodge structure on a
certain curve $C'$. So  $\rho$ virtually projectively
factors through an orbicurve and we can again conclude by  Corollary
\ref{virtually-implies-usually}.
This completes the proof of Theorem 4.1. 
\end{proof}

\medskip

\section{Some remarks on the case $(1,1,...,1)$}
\label{sec-someremarks}

In this section we consider 
the
possibility of generalisation of the
results of Section \ref{structure=(1,1,1)} to the case
$(1,1,...,1)$. For simplicity of notation, we consider only the case
when $X$ is a smooth projective variety (i.e., $\bar X=X$).

Let $n$ be a positive integer and let $(E=\bigoplus_{p+q=n}E^{p,q},
\theta)$ be a rank $(n+1)$ system of Hodge bundles of type
$(1,...,1)$.  Then $\theta$ induces injections
$$L_i:=E^{n-i+1, i-1}\otimes (E^{n-i,i})^*\hookrightarrow
  \Omega _X$$
for $i=1,...,n$.

\begin{Lemma}
Assume that $c_1(E)=0$ and $(E, \theta)$ is slope $A$-stable. Let us set
$$N_i:=(n-i+1)\left(L_1+2L_2+...+iL_{i}\right)+i\left((n-i)L_{i+1}+...+2L_{n-1}+L_n\right)$$
for $i=1,...,n$. Then for all non-negative real numbers $x_1,...,x_n$ we have
$$\left( \sum _{i=1}^nx_iN_i\right) ^2\le 0.$$
\end{Lemma}

\begin{proof}
We claim that for all $i=1,...,n-1$ the map
$$\alpha _i : L_i\oplus L_{i+1}\to \Omega_X$$
induced from $\theta$ has rank $1$. Indeed, in the local coordinates
$\alpha$ is given by two $1$-forms $\omega_1$ and $\omega_2$ and the
integrability of $\theta$ implies that $\omega_1\wedge \omega_2=0$, so
they are proportional and $\alpha_i$ has rank $1$. It follows that
the map
$$\alpha  : \bigoplus_{i=1}^n L_{i}\to \Omega_X$$
also has rank $1$. Let $M$ be the saturation of the image of $\alpha$
in $\Omega_X$ (i.e., the largest rank $1$ subsheaf of $\Omega_X$
containing the image of $\alpha$). Then $M$ is a line bundle such that
$\Omega_X/M$ is torsion free.

By construction we have inclusions $L_i \subset M$ defining effective
divisors $B_i$. Let us set $e^{p,q}=c_1(E^{p,q})$, $l_i=c_1(L_i)$ and
$b_i=c_1(B_i)$. Then we can write
$$m=c_1(M)=l_1+b_1=...=l_n+b_n.$$
By assumption we have
$$e^{n,0}+e^{n-1,1}+...+e^{0,n}= 0$$
in the rational cohomology $H^2(X, \QQ)$, so the classes $e^{p,q}$ can
be written in terms of $l_i$ in the following way:
$$ e^{n-i,i} =\frac{1}{n+1} (-l_1-2l_2-...-i l_i+(n-i)l_{i+1}+(n-i-1)l_{i+2}+...+l_n)
$$
for $i=0,...,n$.

Slope $A$-stability of  $(E, \theta)$ implies that
$$(e^{0,n}+e^{1,n-1}+...+e^{i,n-i}). \, A < 0$$
for $i=0,...,n-1$. This is equivalent to
$$(e^{n,0}+e^{n-1,1}+...+e^{i,n-i}). \, A > 0$$
for $i=1,...,n$.
Rewriting these inequalities in terms of $L_i$ we get
$$ N_i.A>0$$
for $i=1,...,n$.
Since $L_i\le M$ we have
$$\sum _{i=1}^n x_iN_i\le \left(\sum _{i=1}^n \frac{i(n-i+1)(n+1)}{2}x_i\right)M.$$
for all non-negative rational numbers $x_1,..., x_n$.  But by Lemma \ref{not-big},
$M|_Y$ is not big and hence
the $\QQ$-divisor $\sum _{i=1}^nx_iN_i$ restricted to $Y$ is also not
big. Therefore we have
$$\left( \sum _{i=1}^nx_iN_i\right) ^2\le 0$$
for all non-negative rational numbers $x_1,..., x_n$ (so also for all non-negative real numbers $x_1,..., x_n$).
\end{proof}

\medskip

This lemma gives many inequalities. Unfortunately, when $n\ge 3$
vanishing of higher Chern classes of $E$ does not seem to give any
particularly useful information.
For example for $n=3$ in dimension $2$ it gives
$$L_1^2+(L_1+2L_2)^2+(L_1+2L_2+3L_3)^2=0.$$
Since
$$N_1=3L_1+2L_2+L_3,$$
$$N_2=2L_1+4L_2+2L_3,$$
$$N_1=L_1+2L_2+3L_3,$$
$E$ is $A$-stable if $L_2.A$ is large and $L_1.A$ and $L_3.A$ are not too negative. In that case $L_1^2$ can be positive and
we do not get an analogue of Proposition \ref{rank-3-structure}.

\section{Corollaries}
\label{sec-cor}

Recall that our variations of Hodge structure are always considered to
be polarizable. We say that a VHS is {\em of weight $k$} if its Hodge
decomposition takes the form $V=\bigoplus _{p+q=k,p,q\geq 0} V^{p,q}$.

\begin{Proposition}
\label{createVHS}
Let $X$ be a smooth complex quasi-projective variety.
Suppose $\rho _K : \pi _1(X,x)\rightarrow \GL (n, K)$ is an absolutely irreducible representation defined over an algebraic number field, such
that for each embedding $\sigma :K\rightarrow \CC$ the associated $\CC$-local system $V_{\sigma}$ underlies a
polarized complex VHS of weight $k$. Then there is a totally imaginary quadratic extension $L$ of a totally real algebraic number field $F$, with
$L$ Galois over $\QQ$, and
a representation $\rho _{L}: \pi _1(X,x)\rightarrow \GL (n, L)$, together with an extension $K'$ containing both $K$ and $L$
such that the extensions of scalars of $\rho _K$ and $\rho _{L}$ to $K'$ are isomorphic (i.e.\ conjugate). Let $V_{L}$ denote the $L$-local system
corresponding to $\rho _{L}$ and let $V_{L/\QQ}$ denote the same local system considered as a local system of $\QQ$-vector spaces
by restriction of scalars. We may arrange things so that $V_{L/\QQ}$ underlies a polarizable $\QQ$-variation of Hodge structure of weight $k$.
\end{Proposition}
\begin{proof}
This is discussed in  \cite[Section 10]{CS} and \cite[Theorem 5]{HBLS}.
The basic idea goes back at least to Deligne's ``Travaux de Shimura'' \cite{Del}.

The discussion in
the third and following paragraphs of the proof of Theorem 5 in \cite{HBLS} does not use the rigidity hypothesis of that theorem,
but only the statement that for every embedding $\sigma : K\rightarrow \CC$ the induced local system  $V_{\sigma}$ is a polarizable VHS.
One should assume there that $K$ is a Galois extension of $\QQ$ (we can make that hypothesis), and the discussion gives the conclusion
that the subfield $L\subset K$ generated by the traces of monodromy elements, has a uniquely defined complex conjugation operation.

We recall the reason. The field of traces that was denoted by $L$ in \cite{HBLS} will be denoted by $L^{\rm tr}$ here,
and the extension that we are looking for, denoted $L'$ in \cite{HBLS}, will be denoted by $L$ here.

Since $V_{\sigma}$ is polarizable, its complex conjugate is isomorphic to its dual. As we assume $V$ is absolutely irreducible, each $V_{\sigma}$ is irreducible so this isomorphism is unique up to a scalar.

Let $c:\CC \rightarrow \CC$ denote complex
conjugation. It means that $V_{c\circ \sigma} \cong (V_{\sigma})^{\ast}$. Suppose $\gamma \in \pi _1(X,x)$, and let $\rho_K^{\ast}$ denote the
dual representation. Recall that $\rho_K^{\ast}(\gamma ) = (\rho_K (\gamma )^{-1})^t$, so ${\rm Tr}(\rho_K^{\ast}(\gamma ) )
={\rm Tr}(\rho_K (\gamma ^{-1}))$.
From the isomorphism $V_{c\circ \sigma} \cong (V_{\sigma})^{\ast}$ we therefore get
$$
(c\circ \sigma) ({\rm Tr}(\rho _K(\gamma ) ) )= \sigma ({\rm Tr}(\rho _K(\gamma ^{-1}) ) ).
$$
As we are assuming that $K$ is a Galois extension of $\QQ$, the image of $c\circ \sigma$ is equal to the image of $\sigma$ and
it makes sense to write $\sigma ^{-1}\circ c\circ \sigma : K\rightarrow K$ (this notational shortcut was used without explanation in
\cite{HBLS} and required the hypothesis that $K$ be Galois). We obtain
$$
\sigma ^{-1}\circ c\circ \sigma ({\rm Tr}(\rho _K(\gamma ) )) ={\rm Tr}(\rho _K(\gamma ^{-1}) ).
$$
It follows that on the subfield $L^{\rm tr}\subset K$ generated by the traces of monodromy elements, the map
$$
c_{L^{\rm tr}} := \sigma ^{-1}\circ c\circ \sigma :L^{\rm tr}\rightarrow L^{\rm tr}
$$
is well-defined and independent of $\sigma$. In other words, $L^{\rm tr}$ has a well-defined complex conjugation operation.

Therefore, $L^{\rm tr}$ is either
a totally real field, if $c_{L^{\rm tr}}$ is the identity, or a totally imaginary quadratic extension of a totally real field $F:= (L^{\rm tr})^{c_{L^{\rm tr}}}$
otherwise.

Larsen's lemma \cite[Lemma 4.8]{HBLS}
says that we can extend $L^{\rm tr}$ to $L$ (as said before, that was denoted $L'$ there),
a totally imaginary quadratic extension of a totally real field, such that $V$ may be defined over $L$.
The proof uses a consideration of Brauer groups, and we refer to the discussion in \cite[Lemma 4.8]{HBLS}.

We have
$$
V_{L/\QQ} \otimes _{\QQ} \CC = \bigoplus _{\sigma :L\rightarrow \CC} V_{\sigma}.
$$
If $\sigma : L\rightarrow \CC$ then it extends to $\sigma _{K'}:K'\rightarrow \CC$ which then restricts to an embedding
$\sigma _K: K\rightarrow \CC$, and $V_{\sigma} \cong V_{\sigma _K}$. Therefore, our hypothesis that the $V_{\sigma _K}$
are VHS of weight $k$ implies that the $V_{\sigma}$ are VHS of weight $k$.

Our hypothesis says that for any $\sigma$, the Hodge types of $V_{\sigma}$ can be
chosen within $\{ (p,q):\,  p+q=k \mbox{ and }p,q\geq 0\}$.
There could be a choice to be made in assigning the Hodge types, for instance if $V_{\sigma}$ is unitary then it has only a single Hodge type and
that could be put at any $(p,q)$ in the desired set.

As noted in the discussion of \cite[Section 10.2]{CS} we may choose the
Hodge types in such a way that there is a $\QQ$-polarization. We may do the analogue of \cite[Lemma 10.3]{CS} for the
general case as follows.
Let $c_L:L\rightarrow L$ and $c:\CC \rightarrow \CC$ denote the complex conjugation automorphisms.
For $\sigma : L\rightarrow \CC$ put $\overline{\sigma}:= c\circ \sigma = \sigma \circ c_L$.
Choose a collection of embeddings $\sigma _1,\ldots , \sigma _d: L\rightarrow \CC$ such that
the $\overline{\sigma _i}$ are distinct from the $\sigma _j$ and the collection
$$
\{ \sigma _1,\ldots , \sigma _d , \overline{\sigma}_1, \ldots , \overline{\sigma} _d\}
$$
is equal to the full collection of embeddings. For each $i$ choose a structure of VHS for $V_{\sigma _i}$, and then put
$$
V_{\overline{\sigma}_i}^{p,q}:= \overline{V_{\sigma _i}^{q,p}}.
$$
This determines the structures of VHS for $V_{\overline{\sigma}_i}$ in such a way that the property of \cite[Lemma 10.3]{CS}  holds.

As in \cite[Proposition 10.2]{CS}, there is a form
$$
\Phi : V_L\times V_L \rightarrow L
$$
which is $\QQ$-bilinear, and $L$-linear in the first variable and $L$-$c_L$-antilinear in the second variable that is to say
$\Phi (u,\lambda v) = c_L(\lambda )\Phi (u,v)$. Also we may assume that it is $c_L$-$(-1)^k$-symmetric: $\Phi (v,u)= (-1)^kc_L\Phi (u,v)$
by multiplying by a purely imaginary element if necessary to get this sign right.
The form $\Phi$ is unique up to  multiplying by an element of $F^{\ast}$.

Put $\Phi _{\QQ}:= {\rm Tr}_{L/\QQ }\circ \Phi$, giving a $\QQ$-bilinear form
$$
\Phi _{\QQ}:V_{L/\QQ } \times V_{L/\QQ } \rightarrow \QQ
$$
that is $(-1)^k$-symmetric. This will be our polarization form, after having first adjusted $\Phi$ by multiplying by an appropriate
element of $F^{\ast}$ to be chosen below. Let
$$
\Phi _{\CC}:(V_{L/\QQ }\otimes _{\QQ}\CC ) \times (V_{L/\QQ }\otimes _{\QQ}\CC )  \rightarrow \CC
$$
be the induced $\CC$-bilinear form. We would like to say that $\Phi _{\CC}$ pairs $V_{\sigma}$ with
$V_{\overline{\sigma}}$.

To see this, let $v\mapsto z.v$ be the action of $z\in L$ upon $v\in V_{L/\QQ}$.
This extends to an action of $L$ on $V_{L/\QQ}\otimes _{\QQ} \CC$ by $\CC$-linear automorphisms.
Then $V_{\sigma} \subset  V_{L/\QQ}\otimes _{\QQ} \CC$ is the subset of vectors such that
$$
z.v = \sigma (z)v
$$
for $z\in L$.
We have
$$
\Phi (z.u,v)= \Phi (u,  (c_L(z)). v).
$$
Thus, if $u\in V_{\sigma}$ and $v\in V_{\tau}$ then
\begin{equation*}
\begin{split}
\Phi _{\CC} (z.u,v)&=\Phi _{\CC}(\sigma (z)u,v)=\sigma
(z)\Phi _{\CC}(u,v)
= \Phi _{\CC}(z, (c_L(z)).v)
\\ &
=\Phi
_{\CC}(u,\tau (c_L(z))v)
=\tau (c_L(z))\Phi _{\CC}(u,v) = \overline{\tau}(z)\Phi
_{\CC}(u,v).\\
\end{split}
\end{equation*}
Therefore, if $\sigma (z)\neq \overline{\tau} (z)$ then this expression must be zero.
For $\sigma \neq \overline{\tau}$ there is a $z\in L$ with $\sigma (z)\neq \overline{\tau} (z)$ and we obtain $\Phi _{\CC}(u,v)=0$. We conclude that
$\Phi _{\CC}$ pairs $V_{\sigma}$ with $V_{\overline{\sigma}}$.

The complex conjugation on $V_{L/\QQ }\otimes _{\QQ}\CC $ also sends $V_{\sigma}$ to $V_{\overline{\sigma}}$, indeed if
$z.v = \sigma (z)v$ then $z. \overline{v} = \overline{z. v} = \overline{\sigma (z)v} = \overline{\sigma}(z)v$.

The Hodge types on $V_{\CC}:= V_{L/\QQ }\otimes _{\QQ}\CC$ are defined by
$$
V_{\CC}^{p,q}:= \bigoplus _{\sigma} V^{p,q}_{\sigma}.
$$
With our choice of Hodge types satisfying the property of \cite[Lemma 10.3]{CS}, the complex conjugation operation satisfies
$\overline{V_{\CC}^{p,q}} = V_{\CC}^{q,p}$.

We get the hermitian form $(u,v)\mapsto \Phi _{\CC}(u, \overline{v})$ on $V_{\sigma}$ induced by $\Phi _{\QQ}$. As the hermitian forms
on the irreducible local system $V_{\sigma}$ are unique up to multiplication by a real scalar, we deduce that this hermitian form is either a
polarization, or minus a polarization of the complex VHS $V_{\sigma}$. By uniqueness, the hermitian form has a pure Hodge type, and
this type has to be $(k,k)$ otherwise the Hodge types of $V_{\sigma}$ and $V_{\overline{\sigma}}$ could not be symmetric under $(p,q)\mapsto
(q,p)$.

In order for $\Phi _{\QQ}$ to define a polarization to get a $\QQ$-variation of Hodge structure on $V_{L/\QQ }$, we need these forms on $V_{\sigma}$
to define polarizations rather than minus polarizations. This problem of fixing the sign, by multiplying our original form $\Phi$ by an element of
the totally real field $F^{\ast}$, was discussed in \cite[Lemma 10.4]{CS} and we refer there for the proof, which just uses the fact that $F$ is dense
in $F\otimes _{\QQ} \RR$.

Once this adjustment has been made, we obtain a polarization form $\Phi _{\QQ}$, such that with our choice of Hodge decompositions
we obtain a $\QQ$-variation of Hodge structure $V_{L/\QQ }$ of weight $k$.
\end{proof}

{\em Remark:} In the situation of the proposition, notice that the original irreducible local systems $V_{\sigma _K}$
for $\sigma _K : K\rightarrow \CC$, are among the complex direct factors $V_{\sigma}$ of the $\QQ$-variation that was constructed. That was
pointed out during the proof.

\begin{Lemma}
\label{BassSerre}
Suppose in the situation of the previous proposition, that we know furthermore the traces ${\rm Tr}(\rho (\gamma ))$ are algebraic
integers for $\gamma \in \pi _1(X,x)$. Then the $\QQ$-VHS constructed there is a $\ZZ$-VHS.
\end{Lemma}
\begin{proof}
By Bass-Serre theory  \cite{Bass}, there is a projective $\cO _L$-module $P$ such that $V_L = P\otimes _{\cO _L} L$ (speaking here
about the fiber over the basepoint which is a representation of $\pi _1(X,x)$). Then $P$ gives a local system of
free $\ZZ$-modules forming a lattice inside the $\QQ$-local system $V_{L/\QQ}$.
\end{proof}

\begin{Corollary}\label{previous-cor}
  Let $X$ be a smooth complex projective variety.  Suppose $V$ is a
  rank $3$ local system over an algebraic number field $K$ such that:
\begin{enumerate}
\item the monodromy group is Zariski dense in $\SL (3, \overline{K})$;
\item $V$ is integral, in the sense
that ${\rm Tr} (\rho (\gamma ))\in \cO_K$ for any group element
$\gamma$ (here $\rho$ being the monodromy representation), and
\item for any $\sigma : K\rightarrow \CC$ the induced $\CC$-local
  system $V_{\sigma}$ is a VHS.
\end{enumerate}
Then either $V$
  projectively factors through an orbicurve, or else the $V_{\sigma} $
  are direct factors of the monodromy of a family of abelian
  varieties.
\end{Corollary}

\begin{proof}
Assume that $V$ does not projectively factor through an orbicurve. The same holds for each $V_{\sigma}$.
Then by Theorem \ref{main_thm_on_type_(1,1,1)},
we can conclude that the VHS $V_{\sigma}$ are not of type $(1,1,1)$, because the first case of the theorem is
ruled out by our hypothesis that the monodromy group is Zariski dense.

Therefore, the VHS $V_{\sigma}$
may be chosen to have set of Hodge types contained in $\{ (1,0), (0,1)\}$. Indeed, if $V_{\sigma}$ is unitary we can just choose one of them;
if not, it must have Hodge numbers $(1,2)$ or $(2,1)$ since $(1,1,1)$ is ruled out. In either case we can arrange the Hodge types within
$\{ (1,0), (0,1)\}$. In other words, each $V_{\sigma}$ has a structure of VHS of weight $1$.

We may now apply Proposition \ref{createVHS} to conclude that the $V_{\sigma}$ is a direct factor in the monodromy of a
$\QQ$-VHS of weight $1$ denoted $V_{L/\QQ}$. By Lemma \ref{BassSerre} using our hypothesis (1), it is a polarized $\ZZ$-VHS of weight $1$.
This corresponds to a family of abelian varieties \cite{Del}.
\end{proof}

\begin{Remark}
  Let $V$ be a rank $1$ local system defined over an algebraic number
  field $K$, integral and such that for each embedding $\sigma :
  K\rightarrow \CC$ the associated local system $V_{\sigma}$ is a
  polarizable VHS. Then $V_{\sigma}$ is unitary so Kronecker's theorem
  implies that the monodromies of $V$ are roots of unity. In other
  words $V$ is of finite order (cf. proof of Proposition
  \ref{prop-1-9}).
\end{Remark}

\begin{Lemma}\label{rk-1-rigid}
If for some $n\ge 1$ there exists a rigid representation $\rho: \pi_1(X,x)\to \GL (n, \CC)$  then  the
moduli space of rank $1$ local systems is a torsion abelian group.
In particular,  a rigid local system of rank $1$ is of finite order.
\end{Lemma}

\begin{proof}
If the moduli space of rank $1$ local systems has positive dimension then its connected component
of identity is divisible by $n$ and there exists a non-isotrivial family of rank $1$ local systems $W$ containing the trivial system and such that $W^{\otimes n}$ is also a non-isotrivial family. If  $V_{\rho} \otimes W_1\simeq V_{\rho} \otimes W_2$ for some rank $1$ local systems $W_1$ and $W_2$ then the determinants coincide and $W_1^{\otimes n}\simeq W_2^{\otimes n}$. Therefore the family $V_{\rho} \otimes W$ of rank $n$ local systems is non-isotrivial. Since it
contains $V_{\rho}$, we get a contradiction with rigidity of $\rho$.

It follows that $H^1(X,\QQ ) =0$ and the moduli space of rank $1$ local systems is a
torsion abelian group, so for any rigid $\rho$ the system $\det V_{\rho} $ is of finite order.
\end{proof}

\medskip

\begin{Definition}
Let $\rho: \pi_1(X,x)\to \GL (n, \CC)$ be a representation. The \emph{geometric monodromy group}
$M_{\rm geom}(\rho)$ of $\rho$ is the Zariski closure of the image of $\rho$ in $\GL (n, \CC)$. We say that $\rho$
is \emph{properly rigid} if $\rho: \pi_1(X,x)\to M_{\rm geom}(\rho)$ is rigid.
\end{Definition}

Note that any rigid representation is properly rigid. On the other hand, a properly rigid representation might be nonrigid. For example, it could happen that $H^1(X,\QQ )\neq 0$. Therefore, it is more general to consider properly rigid representations.

Given a rigid representation into $\GL (n ,\CC )$, by the previous lemma the identity component $M_{\rm geom}(\rho)^{0}$ of its geometric monodromy group is contained in $\SL (n, \CC)$. In the following we are interested in properly rigid representations with the largest possible geometric monodromy group. More precisely, we are interested in representations $\rho: \pi_1(X,x)\to \GL (n, \CC)$ which satisfy one of the following equivalent conditions:
\begin{enumerate}
\item $M_{\rm geom}(\rho)$ contains  $\SL (n,\CC )$ as a subgroup of finite index,
\item $M_{\rm geom}(\rho)^0=\SL (n,\CC)$,
\item $\mathop{\rm Lie} \, M_{\rm geom}(\rho)=\mathfrak{sl} (n, \CC)$.
\end{enumerate}
If any of  
these
conditions is satisfied then the composition $$\rho: \pi_1(X,x)\mathop{\longrightarrow}^{\rho} \GL (n, \CC)\mathop{\longrightarrow}^{\det} \CC ^*$$
maps onto the group of $m$-th roots of unity $\mu_m$, where $m$ is the index of $M_{\rm geom}(\rho)^0$ in $M_{\rm geom}(\rho)$. Then the surjection $\pi_1(X,x)\rightarrow \mu_m$ gives rise to a degree $m$ \'etale Galois covering $f: Z\to X$ such that $M_{\rm geom}(f^*(\rho))= \SL (n, \CC)$.

\begin{Corollary} \label{rigid+integral+SL}
Let $V$ be an irreducible complex local system  of rank $3$. Assume that the corresponding representation is integral, properly rigid and that the geometric monodromy group of $V$ contains
$\SL (3,\CC )$ as a subgroup of finite index.  Then
either $V$ projectively factors through an orbicurve, or $V$ comes as a direct factor in the monodromy of a family of abelian varieties.
\end{Corollary}

\begin{proof}
Rigidity implies that $V$ is defined over an algebraic number field $K$ and it implies  that for any $\sigma :K\to \CC$ the induced local system $V_{\sigma}$ is a VHS (see \cite[Lemma 4.5]{HBLS}). By the above, we can find a finite
\'etale covering $f: Z\to X$ such that $M_{\rm geom}(f^*(\rho))= \SL (n, \CC)$. Therefore $f^*(\rho)$ satisfies
Hypothesis (1)  of Corollary \ref{previous-cor}. Clearly, $f^*V_{\sigma}$ is also a VHS for any $\sigma$ so
$f^*(\rho)$ satisfies also Hypothesis (3). Hypothesis (2) is provided by our assumption that $V$ is integral.
Hence Corollary \ref{previous-cor} implies that $f^*V$  projectively factors through an orbicurve, or else the $f^*V_{\sigma} $ are direct factors of the monodromy of a family of abelian varieties.

We may descend these conclusions back to $X$. Lemma \ref{biggergroup}
says that if $f^*V$ projectively factors through an orbicurve, then
$V$ does so. If on the other hand $f^*V$ comes from a family of
abelian varieties, then taking the product of the Galois conjugates of the family gives a family which descends back to $X$, and $V$ is a direct factor in the monodromy of this family.
\end{proof}

\begin{Corollary}
If $V$ is any local system of rank $3$ for which the geometric monodromy group contains
$\SL (3,\CC )$ as a subgroup of finite index, then either
it projectively factors through 
an orbicurve, 
or it comes from a family of abelian varieties, or else there is a non-constant equivariant map from the universal cover of $X$ to a two-dimensional building.
\end{Corollary}

\begin{proof}
  If the representation is non-rigid, or non-integral, then there is a
  non-cons\-tant equivariant map from the universal cover of $X$ to a
  two-dimensional building (see \cite{GS} and \cite[Section 4]{CS}). Otherwise, the
  previous corollary applies.
\end{proof}

Let us recall that by Lemma \ref{rk-1-rigid} a rigid local system $V$ has $\det (V)$ of finite order.

\begin{Corollary}\label{geometric-rk-3}
Suppose $V$ is an irreducible rank $3$ local system. If $V$ is integral and properly rigid with $\det (V)$ of finite order, then it is of geometric origin.
\end{Corollary}

\begin{proof}
Let us first assume that the geometric monodromy group contains
$\SL \! (3,\CC )$ as a subgroup of finite index.
If $V$ projectively factors through an orbicurve, then
the factoring representation is rigid.
As was shown
in \cite{CS} using Katz's theorem, it follows that
the factoring representation is of geometric origin. Let us
recall briefly the argument here. If an orbicurve admits a rigid representation, then its coarse moduli space is the projective line and there is an open subset of the orbicurve that is
isomorphic to $\PP^1 - \{ t_1,\ldots , t_k\}$. By modifying the monodromy at one point, we may assume that we have an $\SL \! (3,\CC )$-local system on the orbicurve, pulling back to a local system projectively equivalent to our original one. It is still rigid as a local system on the
orbicurve, which implies also that it is rigid as a local
system on $\PP^1 - \{ t_1,\ldots , t_k\}$ with fixed
(semisimple, finite-order) monodromy transformations at the points $t_i$. Therefore,
Katz's theorem \cite{Ka} says that the new local system has geometric origin.  We may now complete the proof of this case:
the last part of Proposition \ref{prop-1-9} shows that our original representation was of geometric origin.

If $V$ does not projectively factor through an orbicurve then
Corollary \ref{rigid+integral+SL} implies
that $V$ is a direct factor in the monodromy of a family
of abelian varieties, so it is of geometric origin.

If the geometric monodromy group is not a group containing
$\SL (3, \CC)$ then it is either finite, isogenous to
$\SL (2, \CC )$, or isogenous to a positive-dimensional torus.
Note however that in the case of a monodromy
group isogenous to a positive-dimensional torus, the representation is not rigid, indeed a rigid representation to a finite extension of a torus must have finite image. A representation with
finite monodromy group has geometric origin, see Proposition \ref{prop-1-9}.
For monodromy groups isogenous to $\SL (2, \CC )$ the result of \cite{CS} shows that rigid integral local systems are of geometric origin.
\end{proof}

{
\noindent
{\em Proof of Theorem \ref{thm-main1}:}
We note that the statement of Theorem \ref{thm-main1} is
a special case of Corollary \ref{geometric-rk-3}. Indeed,
assume the hypotheses of Theorem \ref{thm-main1} hold, then
since any rigid representation is properly rigid, and
a representation into $\SL \! (3,\CC )$ automatically has
determinant of finite order, the hypotheses of
Corollary \ref{geometric-rk-3} hold so the local system
is of geometric origin. \hfill $\Box$
}

\medskip

We leave it to the reader to formulate some other variants on these corollaries. For instance, one could say
that a rigid integral representation of rank $3$ comes (projectively) by pullback from a curve or a  Shimura modular variety.

The hypothesis of rigidity could be replaced by a hypothesis that
the representation is a point of a $0$-dimensional stratum of a stratification by, say, dimensions of cohomology groups or singularities of the moduli space.

We do not know how to show the conjecture ``rigid implies
integral''. One could hope that in some kind of very heuristic sense
the argument for the $(1,1,1)$ case could motivate us to analyze what
happens when we have a map to a $2$-dimensional building.

\begin{Corollary}
Let  $X$ be a smooth projective variety such that for $i=1,2,3$ we have $H^0(X, \Sym ^i\Omega_X)=0$. Then any
representation $\pi_1(X,x)\to \GL (3, \CC)$ is of geometric origin.
\end{Corollary}

\begin{proof}
If $X$ is a smooth projective variety such that $H^0(X, \Sym ^i\Omega_X)=0$ for $i=1,2,3$,
then any representation $\pi_1(X,x)\to \GL (3, \CC)$ is rigid and integral (see \cite[Theorem 1.6 and Corollary 1.8]{Kl}).
Corollary \ref{geometric-rk-3} implies that it is of geometric origin.
\end{proof}

\medskip

Interesting examples of such varieties come from \cite[Theorem
1.11]{Kl}. Namely we can take the ball quotient $\Gamma
\backslash {\mathbf B^n_{\CC}}$ for a torsion free Kottwitz
lattice $\Gamma \subset SU (n,1)$ for $n\ge 4$ such that $n+1$ is
prime. But these varieties are simple Shimura varieties, so
probably for such varieties our result follows easily in a
different way.

Our results give geometric origin for
two-dimensional ball quotients
{(that is to say, any smooth projective surfaces
of general type with $c_1^2=3c_2$)}, provided we know integrality:

\begin{Corollary}
\label{ballquot}
Suppose $\Gamma \subset PU(2,1)$ is a torsion-free
cocompact lattice
that is integral, i.e.\ such that the traces of its elements are algebraic integers. Let $X=B^2/\Gamma$ be the ball quotient which is a smooth projective variety. Then the standard representation of $\pi _1(X)=\Gamma$ is of geometric origin.
\end{Corollary}
\begin{proof}
The standard representation $\rho$ is {cohomologically rigid
(see below)},
by \cite{Weil}, so all of its conjugates $\rho ^{\sigma}$ for
$\sigma \in {\rm Gal}(\CC / \QQ )$ are
{cohomologically}
rigid, hence they are variations
of Hodge structure. There exists a finite possibly ramified
covering $Z\rightarrow X$ such that $\rho _Z=\rho |_{\pi _1(Z)}$
lifts to $SU(2,1)$ (Lemma \ref{lifting}), in other words it is
a rank $3$ local system. It is clearly Zariski-dense. Furthermore, $\rho _Z$ is a variation of Hodge structure where the period map is the composition
$$
\widetilde{Z}\rightarrow \widetilde{X}\cong B^2,
$$
in particular the differential of this period map at a general point has rank $2$. Therefore, $\rho _Z$ cannot factor through an orbicurve.
Now apply Corollary \ref{previous-cor} to get that $\rho _Z$ has geometric origin, and by Proposition \ref{prop-1-9} $\rho$ is of geometric origin (meaning that its composition with any linear representation of $PU(2,1)$ has geometric origin).
\end{proof}

\begin{Remark}[Added in revision]
\label{rem-eg}
A vector bundle $E$ with an integrable connection $\nabla$ on $X/\CC$ is called \emph{cohomologically rigid} if
$H^1(X, \mathop{\mathcal End} ^0 (E, \nabla))=0$. A local system  is  \emph{cohomologically rigid}  if the corresponding
vector bundle with an integrable connection is cohomologically rigid.
Esnault and Groechenig prove that
a cohomologically rigid
local system is integral \cite{EG}. This means that
in many of our statements above we can remove the integrality
hypothesis if we assume cohomological rigidity.

An important special case is that this allows us to remove
the hypothesis of integrality in Corollary \ref{ballquot}.
Indeed for two-dimensional ball quotients the standard representation
$\Gamma \rightarrow PU(2,1)$ is cohomologically rigid
by \cite{Weil}, therefore \cite{EG} applies to provide the hypothesis of integrality needed to apply Corollary \ref{ballquot}. We conclude that the standard representation
is of geometric origin, for any smooth projective surface
of general type with $c_1^2=3c_2$.
\end{Remark}

\section{Appendix.  The factorization theorem}

The main aim of this appendix is to prove Theorem \ref{factorequiv}.
Before giving proof of this theorem, we need a few auxiliary results.
The first one is a strong version of Proposition \ref{weak-general}
and Remark \ref{weak-general-fibration} in case of representations
into $\SL (n, \CC)$.

\begin{Proposition}
\label{redonfiber}
Let $X$ be a smooth complex quasi-projective variety with an
irreducible representation $\rho :\pi_1(X, x)\to \SL (n, \CC)$.
Suppose $f:X\rightarrow C$ is a fibration 
(cf. \S \ref{sec-notation})
over an orbicurve such
that for a general fiber $F=f^{-1}(c)$, the restriction of $\rho$ to
$F$ is reducible. Then one of the following holds:
\begin{enumerate}

\item $\rho$ projectively factors through $f$.

\item There exists a finite \'etale cover $C_Z\to C$, which by base change induces a finite \'etale cover $p:Z\to X$ such that $p^*\rho$
  is tensor decomposable. In particular, $\rho$ is virtually tensor
  decomposable.

\item There exists a finite \'etale cover $C_Z\to C$, which induces a
  finite \'etale cover $p:Z\to X$ such that $p^*\rho$ is reducible. In
  particular, $\rho$ is virtually reducible.

\end{enumerate}
\end{Proposition}

\begin{proof}
  Some arguments in the following proof are similar to the one used,
  e.g. in \cite[p.~92--93]{Ka}.

Let $V_{\rho}$ denote the local system associated to $\rho$.
Write the isotypical decomposition
$$
V_{\rho}|_F = \bigoplus V_i \otimes W_i ,
$$
where $V_i$ are distinct irreducible representations of $\pi _1(F,x)$
and $W_i$ are vector spaces. By hypothesis, either there is more than
one factor, or there is a single factor but $W_1$ has dimension $>1$.

Use the homotopy exact sequence (Theorem \ref{Xiao}), assuming that
$C$ has the maximal orbifold structure such that $f$ is a morphism,
and letting $x\in F$.  Let $\Phi$ denote the image of $\pi _1(F,x)$ in
$\pi _1(X,x)$, which is also the kernel of $\pi _1(X,x)\rightarrow \pi
_1(C,c)$.  Then $\pi _1(C,c)$ acts by outer automorphisms on $\Phi$ {
  (more precisely, we have a homomorphism from $\pi _1(C,c)$ to the
  group of outer automorphisms of $\Phi$)}. The above isotypical
decomposition is a decomposition of the restriction of $\rho$ to
$\Phi$, and inner automorphisms of $\Phi$ preserve the isomorphism
type of the $V_i$. Therefore, $\pi _1(C,c)$ acts on the set of
isotypical components $\{ V_i\}$ by its conjugation action.

Assume that there is more than one different isotypical component. As there are only finitely many isotypical components, there exists a subgroup of finite index in $\pi_1(C, c)$ that stabilizes each isotypical component.  Therefore there exists a finite \'etale covering $p:Z\rightarrow X$  obtained by base change from a finite \'etale
covering $C_Z\to C$ defined by the above subgroup, such that when
we pull back the above picture to $Z$, the isotypical components are
preserved by the action of $\pi _1(C_Z,c')$.  Then the decomposition
is preserved by the full monodromy action of $\pi _1(Z,z)$. So
$V_{\rho}|_Z$ becomes reducible, and in that case $\rho$ is virtually
reducible.

We may therefore now suppose that there is only a single isotypical
component in the above decomposition, that is to say
$$
V_{\rho}|_F = V_1 \otimes W_1,
$$
where $n_1=\rk V_1\ge 1$ and $m_1=\dim W_1\ge 2$.

\medskip

{

\begin{Lemma}
  There exists a finite \'etale covering $C_Z\to C$ of the
  orbicurve
  $C$ such that for the induced \'etale covering $p:Z\rightarrow X$
  inducing $p_F:F_Z\rightarrow F$, the local system $p_F^{\ast}(V_1)$
  extends to a local system $V'_1$ on $Z$.
\end{Lemma}

\begin{proof}
  Let $\Phi\subset\pi_1(X, x)$ be the image of $\pi_1(F, x)$ and let
  $\varphi:{\Phi} \to \GL (n_1, \CC)$ be the irreducible
  representation corresponding to $V_1$.  The homomorphism $\varphi$
  factors through the almost-simple subgroup $G:=\{A\in \GL (n_1,
  \CC): \, (\det A)^{m_1}=1 \}$.  For each element $g \in \pi_1(X,x)$
  we can consider the representation $\varphi _{g}: \Phi \to G\subset
  \GL (n_1,\CC)$ given by $\varphi _g (h)=\varphi (ghg^{-1})$ and the
  corresponding $\Phi$-module $V_{\varphi_g}$. By definition,
  $V_{\varphi}\otimes W\simeq V_{\varphi_g}\otimes W$ and
  $V_{\varphi}$ is a simple $\Phi$-module.  Therefore there exists a
  non-zero map from one of the simple factors in the Jordan--H\"older
  filtration of $V_{\varphi_g}$ to $V_{\varphi}$. By Schur's lemma,
  this factor is isomorphic to $V_{\varphi}$. But $V_{\varphi}$ and
  $V_{\varphi_g}$ have the same dimension over $\CC$, so
  $V_{\varphi_g}$ is simple and the $\Phi$-modules $V_{\varphi}$ and
  $V_{\varphi_g}$ are isomorphic.  Therefore for each fixed $g\in
  \pi_1(X,x)$, we can choose $A_g$ in $G\subset \GL (n_1,\CC)$ such
  that
$$\varphi (ghg^{-1}) = A_g\varphi (h)A_g^{-1}$$
for all $h\in \Phi$.  The map $\bar\tau:\pi_1(X,x)\to \PGL (n_1, \CC)$
defined by sending $g$ to the class of $A_g$, is a projective
representation extending the projective representation $\pi_1(F, x)\to
\PGL (n_1, \CC)$ associated to $V_1$.

{We distinguish two cases. By \cite[Proposition 5.5]{BN} (see also
  \cite[2.2]{CS}), either the orbicurve is spherical, or it has a
  finite \'etale covering that is a usual curve with infinite
  fundamental group.

In the spherical case, there is a finite \'etale covering of $C$ whose
fundamental group is trivial.  After pulling back to this covering we
are in the situation where $\Phi = \pi _1(X,x)$ and the claim follows.

We may therefore assume that there is a finite \'etale covering of $C$
that is a usual curve not equal to $\PP^1$ or $\AA ^1$. Pulling
everything back to this covering, we may assume that $C$ has no
orbifold structure.}

We can choose an analytic neighbourhood $U\subset C$ of $c=f(x)$ isomorphic to a disk
and such that the induced fibration $f^{-1}(U)\to U$ is trivial in the usual
topology (cf. \cite[Lemma 1.5 A]{No}).  If $C$ is compact then we
choose a point $c'\in U-\{c\}$. The fiber $F'$ of $f$ over $c'$ is
smooth and homeomorphic to $F$. Then we define the Zariski open subcurve
$C^*=C-\{c'\}\subset C$.  In case $C$ is non-complete we set
$C^*=C$. Let us also set $X^*=f^{-1}(C^*)$.  By construction, the
fundamental group $\pi_1(C^*, c)$ is free.

Now for some elements $g_1,...g_s\in \pi_1(X^*,x)$ whose images are
free generators of $\pi_1(C^*,c)$, we can consider the
semidirect product
$$\Phi\ltimes \langle g_1,...,g_s\rangle $$
with $g_i$ acting on $\Phi$ by conjugation. This group is isomorphic to $\pi_1(X^*,x)$.
Then we can extend
$\varphi$ to a representation $\tilde \varphi: \pi_1(X^*,x)\to G$
by setting
$$\tilde \varphi (hg_1^{a_1}...g_s^{a_s})=\varphi (h) A_{g_1}^{a_1}...A_{g_s}^{a_s}.$$
The short exact sequence
$$1\to \mu_l\to G\to  \PGL(n_1,\CC)\to 1,$$
where $l=n_1m_1$, leads to  an exact sequence of (non-abelian) group cohomology
$$ H^1(\pi_1(X); G )\to  H^1(\pi_1(X); \PGL (n_1, \CC) ) \mathop{\to}^{\delta} H^2 (\pi_1(X); \mu _l),$$
where the map $\delta$ sends the 1-cocyle $[\bar \tau]=\{A_g\}_{g\in G}$ to the 2-cocycle
$[ \, \,\,]: \pi_1(X)\times\pi_1(X) \to \mu_l$
given by
$$[g,h]=A_g A_hA_{gh}^{-1} .$$
The sequence above shows that this 2-cocycle is the obstruction to lifting $\bar \tau$ to a
representation $\pi_1(X,x)\to G$.

Let $\Phi ^*\subset \pi_1(X^*,x)$ be the image of $\pi_1(F, x)$. We
claim that the canonical surjection $\Phi ^*\to \Phi$ is in fact an
isomorphism. By van Kampen's theorem $\pi_1(X,x)$ is isomorphic to the
amalgamated product of $\pi_1(X^*,x)$ and $\pi_1(f^{-1}(U),x)\simeq
\pi_1(F,x)$ over $\pi_1(X^*\cap f^{-1}(U),x)\simeq \pi_1 (F,x)\times \ZZ$.
This implies that  $\Phi ^*\simeq \Phi$.

Therefore we  have a commutative diagram
$$\xymatrix{
1\ar[r] & \Phi \ar@{=}[d]\ar[r]& \pi_1(X,x)  \ar[r]&  \pi_1(C,c) \ar[r]&  1\\
1 \ar[r]&  \Phi \ar[r]&  \pi_1(X^*,x) \ar@{->>}[u] \ar[r]&  \pi_1(C^*,c)  \ar@{->>}[u]\ar[r]&  1
}$$
that leads to the commutative diagram
$$\xymatrix{
H^2(\pi_1(X,x); \mu_l)\ar[d]\ar[r]^{\res}& H^2(\Phi ; \mu_l )\ar@{=}[d]\\
H^2(\pi_1(X^*,x); \mu_l)\ar[r]^{\res ^*}& H^2(\Phi ; \mu_l) .\\
}
$$
Note that the class $\delta ([\bar\tau])|_{X^*}\in H^2(\pi_1(X^*,x);
\mu_l)$ is zero, since the representation
$\bar\tau|_{X^*}:\pi_1(X^*,x)\to \PGL (n_1, \CC)$ lifts to
$\tilde\varphi :\pi_1(X^*,x)\to G$. Therefore the above diagram shows
that $\delta ([\bar\tau])$ lifts to a class $\eta$ in the kernel of $\res$.
Now, by the Lyndon--Hochschild--Serre spectral sequence
$$H^p(\pi_1(C,c); H^q (\Phi ;\mu_l))\Rightarrow H^{p+q}(\pi_1(X,x); \mu_l ),$$
so this kernel fits into the commutative diagram
$$\xymatrix{
H^2(\pi_1(C,c); \mu_l)\ar[d]\ar[r]&\ker(\res\ar[r]^-{\epsilon}) \ar[d]& H^1(\pi_1(C,c); H^1(\Phi ; \mu_l)) \ar[d]\\
H^2(\pi_1(C^*,c); \mu_l)\ar[r]&\ker (\res ^*) \ar[r]^-{\epsilon ^*}& H^1(\pi_1(C^*,c); H^1(\Phi ; \mu_l)) .\\
}
$$

Now we need to kill $\epsilon (\eta)$ by passing to a finite \'etale
cover of $C$.  Let us first remark that $A:= H^1(\Phi ; \mu_l)= \Hom
(\Phi ; \mu_l)=\Hom (\Phi/[\Phi, \Phi] ; \mu_l)$ is a finite abelian
group. This is clear since $\pi_1(F,x)$, an hence also $\Phi$, are
finitely generated. Passing to the finite \'etale covering defined by
the finite index subgroup $\bigcap _{a\in A}\pi_1(C,c)_a\subset
\pi_1(C,c)$ obtained by intersecting of all the stabilizers of
$\pi_1(C,c)$-action on $A$, we can assume that $A$ is a trivial
$\pi_1(C,c)$-module.  Then
$$H^1(\pi_1(C,c); A)=\Hom (\pi_1(C,c), A)=\Hom (\pi_1(C,c)/[\pi_1(C,c), \pi_1(C,c)], A)$$
is a finite abelian group. Let us
consider a finite index subgroup of $\pi_1(C,c)$ defined by
$$H:= \bigcap _{\varphi\in \Hom (\pi_1(C,c), A)}\ker \varphi \subset \pi_1(C,c).$$
It is clear from the definition that the induced restriction map
$H^1(\pi_1(C,c), A)\to H^1(H;A)$ is zero, so $H\subset \pi_1(C,c)$ defines
a finite \'etale cover of $C$ which kills the class $\epsilon (\eta)$.
Therefore we can assume that $\eta$ lifts to a class $\tilde \eta \in
H^2(\pi_1(C,c); \mu_l)$.

By assumption $C$ is not spherical, so the universal cover $\tilde C$ of $C$ is contractible.
So the spectral sequence
$$H^p(\pi_1(C,c); H^q(\tilde C, \mu_l))\Rightarrow H^{p+q}(C, \mu_l)$$
degenerates to $H^p(\pi_1(C,c);\mu_l)=  H^{p}(C, \mu_l)$.

If $C$ is not projective then there is nothing to prove as $H^2(C,
\mu_l)=0$.  So we can assume that $C$ is projective of genus $g\ge 1$.
For any finite abelian group $A$ any element $\alpha\in H^2(C, A)$ can
be killed after passing to a finite \'etale cover $\pi: C'\to
C$. This follows from the fact that $H^2(C,A)\simeq A$ and $\pi^*:
H^2(C, A)\to H^2(C', A)$ is multiplication by the degree of
$\pi$ and $\pi_1(C,c)$ has a subgroup of index equal to the order of $A$. But
$\pi_1(C,c)$ contains subgroups of arbitrary finite index, so applying
the above remark to the class $\tilde \eta$, we can find a connected degree
$n$ finite \'etale covering $C_Z\rightarrow C$ such that the
pullback of $\tilde \eta$ is zero in $H^2(\pi_1(C_Z,c'); \mu _l)=H^2(C_Z, \mu_l)$.
So letting $Z$ be the pullback of this covering over $X$ we get
that our class in $H^2(\pi_1(X,x);\mu _l)$ pulls back to the zero class
in $H^2(\pi_1(Z,z); \mu _l)$. Therefore, when pulled back to
$Z$, the representation lifts to $G$.
\end{proof}
}

Let $f_Z$ denote the map from $Z$ to $C_Z$. Let us consider
$$V_1'\otimes \Hom (V_1', V_{\rho}|_Z)\to V_{\rho}|_Z.$$
Note that $V_1'|_{F_Z}$ is irreducible, since $p_F: F_Z\to F$ is an
isomorphism (this is why we need $p:Z\to X$ to be induced from a
finite covering of the curves $C_Z\to C$).  So after restricting the
above map to the fiber $F_Z$ we get an isomorphism.  Since this map is
a map of local systems on $Z$, it must be an isomorphism.

If the rank of $V'_1$ is $\geq 2$ this gives a tensor decomposition
over $Z$ so $\rho$ is virtually tensor decomposable. Finally, if
$V'_1$ has rank $1$ then $V_1$ has rank $1$, so the restriction of
$\rho$ to $\pi _1(F,x)$ projects to the trivial representation in
$\PGL (n,\CC )$. This means that $\rho$ projectively factors through
$f$. This finishes proof of Proposition \ref{redonfiber}.
\end{proof}

\medskip

The proof of the next lemma uses projectivity of the base manifold.

\begin{Lemma}\label{virtual-implies-usual}
Let $X$ be a smooth complex projective variety of dimension $\ge 2$. Let
$\rho:\pi_1(X, x)\to \SL (n, \CC)$ be a representation with an infinite image.
Assume that $\rho$ virtually projectively factors through an orbicurve.
Then there exists a fibration $h: X\rightarrow B$ over a smooth projective curve $B$
such that the restriction of $V_{\rho}$ to a general fibre of $h$ has finite monodromy.
\end{Lemma}

\begin{proof}
By assumption there exists an alteration $p:Z\rightarrow X$ for which
$p^*\rho$ projectively factors through a fibration $f:Z\to C$ over an orbicurve $C$.
Let $Z_y\subset Z$ denote the fiber over $y\in C$. Choose $a\in C$ at which the stabilizer
group is trivial and the fiber $Z_a$ is smooth.

Let us consider $g=(p,f): Z\to X\times C$. Let $Y\subset X\times C$ be
the image of $g$ and let $r_1:Y\to X$ and $r_2:Y\to C$ denote the
corresponding projections.

\medskip

\noindent
\emph{Claim.} We have $p(Z_a)\cap p(Z_y) =\emptyset$ for general $y\in C$.

\begin{proof}
Assume that $p(Z_a)\cap p(Z_y) \ne \emptyset$ for general $y\in C$.
Then $Y\cap (p(Z_a)\times C)$ contains an irreducible component that dominates $C$.
So we can choose a complete irreducible curve $D\subset Z$ such that $g(D)$ is
contained in $Y\cap (p(Z_a)\times C)$ and $f(D)=C$.  Let $\nu:\tilde D
\to D\hookrightarrow Z$ be composition of the normalization of $D$
with the canonical inclusion. Then $\tilde D$ is a smooth projective
curve, $p\circ \nu$ maps $\tilde D$ into $p(Z_a)$ but $f\circ \nu: \tilde
D\rightarrow C$ is surjective.
The situation can be summed up in the following diagram:
$$
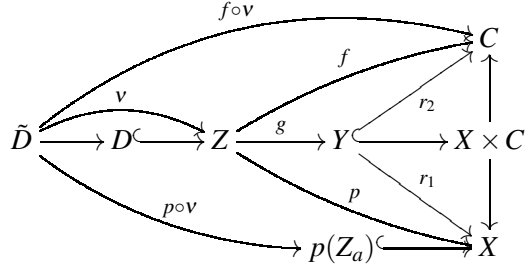
\xymatrix{
&&&&C\\
\tilde D\ar@/^1pc/[rr]^{\nu} \ar[r]\ar@/^2pc/[urrrr]^{f\circ \nu}
\ar@/_1pc/[drrr]^{p\circ \nu}
& D\ar@{^{(}->} [r]& Z \ar[r]^g\ar@/^/[rru]^f\ar@/_/[rrd]^p&Y\ar@{^{(}->} [r]\ar[rd]^{r_1} \ar[ru]_{r_2} &X\times C\ar[d]\ar[u]\\
&&&p(Z_a)\ar@{^{(}->}[r]&X\\
}$$

If $p(D)$ is a point then $\nu^{\ast}p^{\ast}(V_{\rho})$ is a local
system with finite (in fact, trivial) monodromy on $\tilde D$. If
$D':=p(D)$ is a curve then we can choose an irreducible curve
$D''\subset Z_a$ mapping onto $D'$.  Since $p^*\rho$ projectively
factors through $f$, the monodromy of the pullback of $V_{\rho}$ to
$Z_a$ is contained in the center of $\SL (n, \CC)$, so it is
finite. Therefore the monodromy of the pullback of $V_{\rho}$ to the
normalization $\widetilde{D''}$ of $D''$ is also finite.  Since
$\widetilde{D''}\to Z\to X$ factors through the normalization
$\widetilde{D'}$ of $D'$, the image of $\pi_1(\widetilde{D''})$ in
$\pi_1(\widetilde{D'})$ has finite index. Therefore the local system
$V_{\rho}$, pulled back to the normalization $\widetilde{D'}$, has
finite monodromy.  The map $p\circ \nu : \tilde D\to X$ factors
through $\widetilde{D'}$. Therefore, also in this case
$\nu^{\ast}p^{\ast}(V_{\rho})$ is a local system with finite monodromy
on $\tilde D$.

On the other hand, the image of $\pi_1(\tilde D)$ in $\pi_1(C)$ has
finite index (since $f\circ g:\tilde D\to C$ is surjective), the image
of $\rho$ is infinite and $p^*\rho$ projectively factors through $f$. Therefore
the image of $\pi_1(\tilde D)$ in $SL(n,\CC)$ cannot be finite, a
contradiction.
\end{proof}

\medskip

By the above claim, we can choose three distinct, smooth fibres
$Z_{a}$, $Z_{b}$, $Z_{c}$ of $f$ such that their images in $X$ are
pairwise disjoint divisors. Since $Z_a$, $Z_b$ and $Z_c$ are
algebraically equivalent, the corresponding cycles $p_*Z_a$, $p_*Z_b$
and $p_*Z_c$ are also algebraically equivalent. In particular, the
corresponding cohomology classes of $p(Z_a)$, $p(Z_b)$ and $p(Z_c)$
lie on the same line in $H^2(X, \QQ)$. Therefore by \cite[Theorem
2.1]{To} there exists a fibration $h:X\rightarrow B$ over a smooth
curve $B$ such that $p(Z_a)$, $p(Z_b)$ and $p(Z_c)$ are fibers of
$h$. Moreover, by \cite[Lemma 2.2]{To} for general $y\in C$, the image
$p(Z_y)$ is a fiber of $h$. In particular, $p(Z_y)$ is smooth for
general $y\in C$.  Since the monodromy group of the pull back of
$V_{\rho}$ to $Z_y$ is finite and $\pi_1(Z_y)$ has finite index in
$\pi_1(p(Z_y))$, we conclude that the restriction of $V_{\rho}$ to a
general fiber of $h$ has finite monodromy.
\end{proof}

\medskip
Now we can prove Theorem \ref{factorequiv}, which we recall:

\begin{Theorem}
\label{factorequiv2}
Let $X$ be a smooth complex projective variety. Let us fix an
irreducible representation $\rho:\pi_1(X, x)\to \SL (n, \CC)$ for some
$n\ge 2$. Suppose that $\rho$ is not virtually tensor decomposable,
and not virtually reducible. Then the following conditions are
equivalent.
\begin{enumerate}

\item $\rho$ projectively factors through an orbicurve;

\item $\rho$ virtually projectively factors through an orbicurve;

\item There exists a nonconstant
map $f: X\rightarrow C$ to an orbicurve and a smooth
fiber $F=f^{-1}(y)$, such that the restriction of
$\rho$ to $\pi _1(F,x)$ becomes reducible;

\item There exists an alteration $p:Z\rightarrow X$
such that the previous condition holds for the pullback $p^{\ast}\rho$.

\end{enumerate}
\end{Theorem}

\begin{proof}

  Let us first remark that the image of $\rho$ is infinite. Indeed,
    if the image of $\rho$ is finite then there exists a finite
    \'etale covering $X'\to X$ such that the pull back of $\rho$ to
    $X'$ is trivial, contradicting our assumption that $\rho$ is not
    virtually reducible. In the following we assume that $X$ has
    dimension $\ge 2$, as otherwise the theorem is trivial.  Clearly,
    the implications $(1)\Rightarrow (2)$ and $(3)\Rightarrow (4)$ are
    trivial.  Let us assume $(2)$. Then by Lemma
    \ref{virtual-implies-usual} there is a map $h:X\rightarrow B$ to
    an orbicurve, such that the restriction of $V_{\rho}$ to a general
    fiber of $h$ has finite monodromy. Apply the homotopy exact
    sequence (Theorem \ref{Xiao}), making sure that the orbicurve
    structure of $B$ is the necessary one so that it works. Let $F$
    denote a general fiber. Assume that the basepoint $x$ belongs to
    $F$ and denote by $b:= f(x)$ the basepoint in $B$. We get the
    exact sequence $$ \pi_1(F,x) \rightarrow \pi _1(X,x) \rightarrow
    \pi _1(B,b) \rightarrow 1.  $$ Now we need the following
    group-theoretic lemma:

\begin{Lemma} If $H\subset \SL (n, \CC)$ is
      finite and irreducible then the normalizer of $H$ in $\SL (n,
      \CC)$ is finite.  \end{Lemma}

\begin{proof} Since $H$ is a
      closed subgroup of $\SL (n, \CC)$, its normalizer $G:=N_H\SL (n,
      \CC)$ is also closed in $\SL (n , \CC)$. Let us fix an element
      $h\in H$.  Since the identity component $G^{0}$ normalizes $H$,
      the map $\varphi_h:G^{0}\to \SL (n, \CC)$ defined by
      $\varphi_h(g)=ghg^{-1}h^{-1}$ has image in $H$. Since $G^{0}$ is
      connected and $H$ is finite, $\varphi_h$ is constant.  Hence
      $G^{0}$ commutes with $H$. Since $H$ is irreducible, elements
      commuting with $H$ are contained in the centre $Z(\SL (n, \CC))$
      of $\SL (n, \CC)$.  Therefore $G^0$ is finite. Since $G^{0}$ is
      of finite index in $G$, $G$ is also finite.
\end{proof}

Let
    $\Phi\subset \pi _1(X,x)$ denote the image of $\pi _1(F,x)$.
    Suppose that the restriction $\varphi := \rho |_{\Phi}$ is
    irreducible and let $H\subset SL (n, \CC)$ denote the image of
    $\varphi$.  Since $\Phi$ is a normal subgroup of $\pi _1(X,x)$,
    for any $s\in \pi _1(X,x)$, $\rho (s)$ is contained in the
    normalizer of $H$ in $\SL(n,\CC)$.  Therefore, by the above lemma,
    the image of $\rho$ is finite, contradicting our hypothesis.

    Thus the restriction $V_{\rho}|_F$ to a general fiber is
    reducible, proving the implication $(2)\Rightarrow (3)$.

    Now let us assume $(3)$. Passing to the Stein factorization and
    using Theorem \ref{Xiao}, we can assume that $f$ is a fibration
    over an orbicurve, and the restriction $V_{\rho}|_F$ to a general
    fiber of $f$ is reducible. In view of the hypothesis that $\rho$
    is not virtually tensor decomposable and not virtually reducible,
    Proposition \ref{redonfiber} implies that $\rho$ projectively
    factors through $f$.  This completes the proof of the implication
    $(3)\Rightarrow (1)$.

  The fact that $(3)\Rightarrow (1)$ shows that $(4)\Rightarrow
  (2)$. Indeed, in the situation of $(4)$, the pullback of $\rho$ to
  $Z$ satisfies all the hypotheses: it is irreducible because we
  assumed that $\rho$ was not virtually reducible, and the hypotheses
  of not being virtually tensor decomposable and not being virtually
  reducible are conserved. Over $Z$ we are in the situation of $(3)$
  so we have shown that $\rho |_Z$ factors through an orbicurve. This
  shows $(4)\Rightarrow (2)$ which completes the proof.
\end{proof}

\medskip

In order to optimize the statements of our corollaries in Section
\ref{sec-cor}, we need a result concerning groups lying between $\SL(3,\CC )$ and $\GL(3,\CC )$.
For notational reasons we do this separately here, rather than modifying the previous discussion to take into account such groups.  However, one can easily check that the proofs of Proposition \ref{redonfiber} and Theorem \ref{factorequiv2} work also in this more general case with just  minor notational changes.

Suppose $G\subset \GL(3,\CC )$ is a subgroup such that
$\SL(3,\CC )\subset G$ is of finite index. Let $\mu_m:= G/\SL(3,\CC )$ denote
the quotient group which is a finite cyclic group of some order $m$.

\begin{Lemma}
\label{biggergroup}
Suppose $\rho : \pi _1(X,x)\rightarrow G$ is a representation
with Zariski-dense image.
Let $Z\rightarrow X$ be the cyclic covering determined by the composed
representation $\pi _1(X,x)\rightarrow \mu_m$, and let
$$
\rho _Z: \pi _1(Z,z)\rightarrow \SL(3,\CC )
$$
denote the pullback representation. If $\rho _Z$ projectively factors through a map to an orbicurve then $\rho$ projectively factors through a map to an orbicurve.
\end{Lemma}

\begin{proof}
Consider the composed representation
$$
\pi _1(X,x)\rightarrow G \rightarrow \PGL (3,\CC )
\stackrel{{\rm ad}}{\rightarrow} \SL (8,\CC ),
$$
where the last arrow is the adjoint action on the Lie algebra
$\mathfrak{sl}(3)\cong \CC ^8$.
Call this representation $\zeta$. Our hypothesis
tells us that the monodromy group of $\zeta$ is
$\PGL (3,\CC )$, and that its restriction $\zeta _Z$
to $Z$ factors through an orbicurve. As $\PGL(3,\CC )$
has no finite-index subgroups, the monodromy group of any pullback
to a finite cover is still $\PGL(3,\CC )$. In particular, $\zeta$
is not virtually reducible, and not virtually tensor decomposable.
By Theorem \ref{factorequiv2} we conclude that $\zeta$ projectively factors through an orbicurve. However, the composed map
$$
\PGL (3,\CC )\rightarrow \SL (8,\CC )\rightarrow \PGL (8,\CC )
$$
is injective. Therefore, $\zeta$ factors through an orbicurve.
This factorization gives the projective factorization of $\rho$
through an orbicurve, that we are looking for.
\end{proof}

\section*{Acknowledgements}

The authors would like to thank the referees for their remarks.

The first author was partially supported by Polish National Centre (NCN) contract number
2015/17/B/ST1/02634. This collaboration started during the first author's visit to
Universit\'e de Nice -- Sophia Antipolis in May--June 2015. The visit was supported by Szolem Mandelbrojt
prize awarded to the first author by French Mathematical Society (SMF) and Institute Fran\c{c}ais
in Poland.

The second author was partially supported by
French ANR grants TOFIGROU (ANR-13-PDOC-0015)
and Hodgefun (ANR-16-CE40-0011).

\end{document}